\documentclass[9pt]{amsart}

\usepackage{amsthm,amsmath,amssymb,amsfonts,epsfig,bbold}

\newtheorem{lemma}{Lemma}
\newtheorem{proposition}[lemma]{Proposition}
\newtheorem{theorem}[lemma]{Theorem}
\newtheorem{corollary}[lemma]{Corollary}
\newtheorem{definition}{Definition}
\newtheorem{remark}{Remark}
\theoremstyle{definition}
\newtheorem{example}{Example}
\newcommand{\A}{{\mathbb A}}
\newcommand{\B}{{\mathbb B}}
\newcommand{\N}{{\mathbb N}}

\newcommand{\R}{{\mathbb R}}
\newcommand{\C}{{\mathbb C}}
\newcommand{\D}{{\mathbb D}}
\newcommand{\cB}{\mathcal{ B}}

\newcommand{\cE}{\mathcal{ E}}
\newcommand{\cF}{\mathcal{ F}}

\newcommand{\cH}{\mathcal{ H}}
\newcommand{\cK}{\mathcal{ K}}
\newcommand{\cL}{\mathcal{L}}

\newcommand{\cM}{\mathcal{ M}}
\newcommand{\cN}{\mathcal{ N}}
\newcommand{\cT}{\mathcal{ T}}
\newcommand{\Dom}{{\rm Dom}}

\newcommand{\Spec}{{\rm Spec}}
\newcommand{\Span}{{\rm span}}

\renewcommand{\Re}{{\rm Re}\;}
\renewcommand{\Im}{{\rm Im}\;}

\newcommand{\dist}{{\rm dist}}

\newcommand{\dsc}{\mathrm{disc}}
\newcommand{\ess}{\mathrm{ess}}
\newcommand{\rank}{\mathrm{Rank}}

\newcommand{\dis}{\mathrm{dis}}
\newcommand{\wto}{\rightharpoonup}
\renewcommand{\emptyset}{\varnothing}
\renewcommand{\chi}{\mathbb{1}}
\begin{document}
\title{On the convergence of second order spectra and multiplicity}
\author{Lyonell Boulton}
\address{Department of Mathematics and Maxwell Institute for Mathematical 
Sciences, Heriot-Watt University, Edinburgh
EH14 4AS, United Kingdom}
\email{L.Boulton@hw.ac.uk}

\author{Michael Strauss}
\address{Department of Mathematics and Statistics, University of Strathclyde, Glasgow G1 1XH, United Kingdom}
\email{Michael.Strauss@strath.ac.uk}

\date{May 2010}
\maketitle
\begin{abstract}
{\noindent Let $A$ be a self-adjoint operator acting on a Hilbert space.
The notion of second order spectrum of $A$ relative to a given finite-dimensional
subspace $\cL$ has been studied recently in connection with 
the phenomenon of spectral pollution in the Galerkin method. We establish in this paper a general framework allowing us to determine how the second order spectrum encodes precise information about the multiplicity of the isolated eigenvalues of $A$. Our theoretical findings are supported by various numerical experiments on the computation of inclusions for eigenvalues of benchmark differential operators via finite element bases.
\\\\
Keywords: second order spectrum, convergence to the spectrum, spectral pollution, projection methods, Galerkin method.\\\\
2000 Mathematics Subject Classification: 65N12, 37L65.}
\end{abstract}

%%%%%%%%%%%%%%%%%%%%%%%%%
\tableofcontents

\pagebreak

\section{Introduction}

Let $A$ be a self-adjoint operator acting on an infinite
dimensional Hilbert space $\mathcal{H}$ and let $\lambda$ be an isolated
eigenvalue of $A$. For $\mathcal{I}\subset \mathbb{R}$ let
\[
  \chi_{\mathcal{I}}(A)=\int_{\mathcal{I}}\, \mathrm{d} E_\mu
\]
where $E_\mu$ is  the spectral measure associated to $A$.
The numerical estimation of $\lambda$ whenever
\begin{equation} \label{pol_cond}
   \operatorname{Tr} \chi_{(-\infty,\lambda)}(A)
   =\operatorname{Tr} \chi_{(\lambda,\infty)}(A)=\infty,
\end{equation}
constitutes a serious challenge in computational spectral theory. Indeed, it is well established that classical approaches, such as the Galerkin method, suffer from variational collapse under no further restrictions on the approximating space
and therefore might lead to spectral pollution \cite{ar,bbg,bdg,ds,ha,lese,pok,rssv}.

The notion of second order relative spectrum, originated from \cite{da0} (see Definition~\ref{sosp} below), and has recently allowed the formulation of a general pollution-free strategy for eigenvalue computation. This was proposed in \cite{shar,lesh} and subsequently examined in \cite{bo1,bo2,bost,str}. Numerical implementations of the general principle very much preserve the spirit of the Galerkin method and have presently been tested on applications from  Stokes systems \cite{lesh}, solid state physics \cite{bole}, magnetohydrodynamics \cite{str} and relativistic quantum mechanics \cite{bobo}.

In this paper we examine further the potential role of this pollution-free technique for robust computation of spectral inclusions. Our goal is two-folded. On the one hand, we establish various abstract properties of limit sets of second order relative spectra. On the other hand, we report on the outcomes of various numerical experiments. Both our theoretical and practical findings indicate that second order spectra provide reliable information about the multiplicity of any isolated eigenvalue of $A$.

Section~\ref{basics} is devoted to reformulating some of the concepts from \cite{shar,lesh,bo1,bo2,bost} allowing a more general setting. In this framework, we consider the natural notion of algebraic and geometric multiplicity of second order spectral points (Definition~\ref{multip}) and establish a ``second order'' spectral mapping theorem (Lemma~\ref{l:mt}).

In Section~\ref{near_real} we pursue a detailed analysis of accumulation points of the second order relative spectra on the real line. Our main contribution (Theorem~\ref{thm4})  is a significant improvement upon similar results previously found in \cite{bo1,bo2}. It allows calculation of rigourous convergence rates when the test subspaces are generated by a non-orthogonal basis. Concrete applications include the important case of a finite element basis which was not covered by \cite[Theorem~2.1]{bo2}. Our present approach relies upon an homotopy argument which yields a precise control on the multiplicity of the second order spectral points. The argument is reminiscent of the method of proof of Goerisch Theorem; \cite{Goerisch:1987,plum:1990}.  

Theorem~\ref{thm4} also determines the precise manner in which second order spectra encode information about the multiplicity of points in the spectrum of $A$.
When an approximating space is ``sufficiently close'' to the eigenspace corresponding to an eigenvalue of finite multiplicity, a finite set of conjugate pairs in the second order spectrum becomes ``isolated'' and clusters near the eigenvalue. It turns out that the total multiplicity of these conjugate pairs exactly matches the multiplicity of the eigenvalue. This indicates that second order spectra detects in a reliable manner points in the discrete spectrum and their multiplicities, even under the variational collapse condition \eqref{pol_cond}. In Section~\ref{applica} we examine the practical validity of this statement on benchmark differential operators for subspaces generated by a basis of finite elements. 

The final section is aimed at finding the minimal region in the complex plane where the limit of second order spectra is allowed to accumulate. 
It turns out that, modulo a subset of topological dimension zero, this minimal region is completely determine by the essential spectrum of $A$. This gives an insight
on the difficulties involving the problem of finding conditions on the test subspaces 
to guarantee convergence to the essential spectrum.

\subsection{Notation}
Below we will denote by $\Dom(A)$ the domain of $A$ and by
$\Spec(A)$ its spectrum. We decompose $\Spec(A)$ in the standard
manner as the union of essential and discrete spectrum:
\begin{align*}
&\Spec_{\textrm{ess}}(A) := \{\lambda\in \Spec(A)\,:\, \exists x_n\in \Dom(A),\,
\|x_n\|=1,\, x_n\wto 0,\, (A-\lambda)x_n\to 0\},\\
&\Spec_{{\dis}}(A) :=
\Spec(A)\backslash\Spec_{\textrm{ess}}(A).
\end{align*}
The discrete spectrum is the set of isolated eigenvalues of finite multiplicity
of $A$.

Let $(a,b)\subset \R$. Below $\mathbb{D}(a,b)$ will be the open disc
in the complex plane with centre $(a+b)/2$ and radius $(b-a)/2$, and
$\D[a,b]=\overline{\D(a,b)}$. We allow $a=-\infty$
or $b=+\infty$ in the obvious way to denote half-planes or the whole of $\C$.
For $b=a$, $\D[a,a]=\{a\}$ and $\D(a,a)=\varnothing$.

Let $\cK$ be a Hilbert space. Let $\Omega \subset \cK$ be an arbitrary subset and let $\cB\subset \cK$ be a finite subset. We will denote
\[
\dist_{\cK}(\cB,\Omega):=\max_{u\in \cB}\inf_{v\in \Omega} \|u-v\|_\cK
\qquad \text{and} \qquad \dist_{\cK}(u,\Omega)=\dist_{\cK}(\{u\},\Omega).
\]
Here we include the possibility of $\cK \equiv \C$ and write $\dist(u,\Omega)=\dist_\C(u,\Omega)$.

Let $p\in \N\cup\{0\}$. We endow $\Dom(A^p)$ with the graph inner product
\[
      \langle u,v\rangle_{p}:=\sum_{q=0}^p \langle A^qu,A^qv\rangle\quad\forall u,v\in \Dom(A^p)
\]
so that $(\Dom(A^p),\langle\cdot,\cdot\rangle_{p})$ is a Hilbert space with the associated norm denoted by $\Vert\cdot\Vert_p$.
Below we will consider sequences of finite-dimensional subspaces $(\cL_n)\equiv(\cL_n)_{n\in \N}$ growing towards $\Dom(A^p)$ with a given density property determined as follows:
\[
   \Lambda_p=\{(\cL_n)\subset \Dom(A^p):  \forall f\in
   \Dom(A^p),\, \dist_{\Dom(A^p)}(f,\cL_n)\to 0 \}.
\]
Note that if $\|A\|<\infty$, then $\Lambda_0=\Lambda_p$ for any $p\in \N$.
If $(\cL_n)\subset \Lambda_p$ and $\cL_n\subset \cL_{n+1}$, then
$\cup_{n=1}^\infty \cL_n$ is dense in the graph norm of $A^p$.

For a family of closed subsets $\Omega_n\subseteq \C$, the limit set of this family is defined as
\[
     \lim_{n\to \infty} \Omega_n=\{z\in \C:\exists z_n\in \Omega_n,\,
     z_n\to z\}=\{z\in \C:\dist(z,\Omega_n)\to 0\}.
\]
The limit set of a family of closed subsets is always closed.

\begin{remark}
Below we will establish properties of $\lim_{n\to\infty}\Spec_2(A,\cL_n)$ for
sequences $(\cL_n)\in \Lambda_p$. By similar arguments to those presented in this paper, one can establish analogies to all these properties for the ``weak'' limit set
\[
     \bigcap _{n=1}^\infty \overline{\bigcup _{m=n}^\infty
     \Spec_2(A,\cL_m)}
\]
if we impose that for all $f\in \Dom(A^p)$ exists a subsequence $f_{n(k)}\in \cL_{n(k)}$ such that $\|f_{n(k)}- f\|_p\to 0$.
\end{remark}

Let  $b_1,\dots,b_n$ be basis for a subspace $\mathcal{L}\subset \cH$.
Without further mention, we will identify the elements $u\in\mathcal{L}$
with a corresponding $\underline{u}\in\C^n$ in the standard manner:
\[
    u=\sum_{k=1}^n \langle u,b_j^\ast \rangle b_j\qquad \text{and}
    \qquad \underline{u}=(\langle u,b_1^\ast\rangle,\dots,\langle u,b_n^\ast\rangle),
\]
where $\{b_j^\ast\}$ is the basis conjugate to $\{b_j\}$.
If the $b_j$ are mutually orthogonal and $\|b_j\|=1$,
then $b_j^\ast=b_j$. Below we will denote the orthogonal projection onto $\cL$ by $P:\cH\longrightarrow \cL$. When referring to sequences of subspaces $(\cL_n)$ we will use $P_n$ instead.
Note that $(\cL_n)\in \Lambda_0$ iff $P_n\to I$ in the strong operator topology.

\subsection{Toy models of spectral pollution in the Galerkin method}
We now present a series of examples which will later illustrate some of
the results below. See \cite[Theorem~2.1]{lesh} and \cite[Remark~2.5]{lese}.

\begin{example} \label{ex1}
Let $\cH=\Span\{e_n^\pm\}_{n=1}^\infty$ where $e_n^\pm$ is an orthonormal set of vectors and let $\cL_n=\Span\{e^\pm_1,\ldots,e^\pm_{n-1},e^-_n\}$.
Let $f^\pm_n=\frac{1}{\sqrt{2}}(e_n^+\pm e_n^-)$ and define
\[A=\sum_{n\in\N} n|f^+_n\rangle \langle f^+_n|-\sum_{n\in\N} n|f^-_n\rangle \langle f^-_n|\] in its maximal domain.  Then
\[
   \Spec(P_n A\upharpoonright \cL_n)=\{0,\pm 1,\ldots,\pm(n-1)\}
\]
and so
\begin{equation} \label{exa}
    \lim_{n\to \infty}\Spec(P_n A\upharpoonright \cL_n)\setminus \Spec(A)=\{0\}\not=\varnothing.
\end{equation}
Note that the resolvent of $A$ is compact.
\end{example}

\begin{example}\label{com}
If $A$ is strongly indefinite and it has a compact resolvent, then there exists a sequence $(\mathcal{L}_n)\in\Lambda_p$ for all $p\in\mathbb{N}$ such that
\begin{equation} \label{totpol}
\lim_{n\to\infty}\Spec(P_nA\upharpoonright\cL_n) = \mathbb{R}.
\end{equation}
Indeed, let $\Spec(A)=\{\lambda_m^\pm\}_{m\in\N}$
where the eigenvalues are repeated according to multiplicity with $\lambda_{m+1}^-\le\lambda_m^-<0$ and $0\le\lambda_{m}^+\le\lambda_{m+1}^+$.
Let $e_m^\pm$ be eigenvectors associated to $\lambda^\pm_{m}$ and assume that $\{e^\pm_m\}$ is an orthonormal set. Let $\{\gamma_m\}_{j\in\N}$ be an ordering of $\mathbb{Q}$. For each $n\in\mathbb{N}$ choose $n<k\in\mathbb{N}$ with $\lambda^-_k<\gamma_j<\lambda^+_k$ for $1\le j\le n$. Let $\theta_{j}\in (-\pi,\pi]$ be such that
\[
    \gamma_j=\cos^2(\theta_{j})\lambda^-_{k+j}+\sin^2(\theta_{j})
    \lambda^+_{k+j}\quad\textrm{for}\quad1\le j\le n
\]
and define
\[
     f_{nj}=\cos(\theta_{j}) e^-_{k+j}+\sin(\theta_{j}) e^+_{k+j},
\]
then \eqref{totpol} holds for the subspaces
\[
\mathcal{L}_n\equiv\Span\{e^\pm_1,\ldots,e^\pm_n,f_{n1},\ldots,f_{nn}\}
 \subset \Dom(A^p).
\]
\end{example}

\begin{example} \label{ex3}
Let $\cH$ and $\cL_n$ be as in Example~\ref{ex1}.
Let
\[
f^\pm_n=\sin(1/n)e_n^\mp \pm \cos(1/n)e_n^\pm.
\]
For $r>0$, define
$A=\sum_{n\in\N} n^r |f^+_n\rangle \langle f^+_n|-\sum_{n\in\N} |f^-_n\rangle \langle f^-_n|$ in its maximal domain. Then
\begin{gather*}
   \Spec_\ess(A)=\{-1\} \qquad\text{and} \qquad \Spec_\dis(A)=\{n^r\}_{n\in\N}.
\end{gather*}
It is not difficult to see that now
\[
   \Spec(P_n A\upharpoonright \cL_n)=\{-1,n^r\sin(1/n)^2-\cos(1/n)^2,1,\ldots,(n-1)^r\}
\]
where $-1$ is an eigenvalue of multiplicity $n$. Then, if $r=2$,
\eqref{exa} holds true. Note that the resolvent of $A$ is not compact and $A$ is now semi-bounded below.
\end{example}

%%%%%%%%%%%%%%%%%%%%%%%%%%%%%%%%

\section{The second order spectra of a self-adjoint operator}
\label{basics}

\begin{definition}[See \cite{lesh}] \label{sosp} Given a subspace $\mathcal{L}\subset \Dom(A)$, the second order spectrum of $A$ relative to $\cL$ is the set
\[
   \Spec_2(A,\cL):=\{z\in \C\,:\,\exists u\in \cL\backslash\{0\}
,\,\langle (A-zI)u,(A-\overline{z}I)v\rangle =0\quad\forall v\in \cL\}.
\]
\end{definition}
If $\cL\subset \Dom(A^2)$, then
\begin{align*}
    \Spec_2(A,\cL)&=\{z\in \C:\exists u\in \cL\backslash\{0\},\langle P(A-z)^2 u,v\rangle=0
    \quad \forall v\in \cL\} \\
      &=\{z\in \C: 0\in \Spec(P(A-z)^2\upharpoonright \cL)\}.
\end{align*}
Typically $\Spec_2(A,\cL)$ contains non-real points. From the definition it is easy to see that  $z\in \Spec_2(A,\cL)$ if and only if $\overline{z}\in \Spec_2(A,\cL)$.

\subsection{Algebraic and geometric multiplicity}
Let \[\cL=\Span\{b_j\}_{j=1}^n\subset \Dom(A)\] where the $b_j$ are linearly independent. Let $B,L,M\in \C^{n\times n}$ be matrices with entries
given by
\begin{equation}\label{matrices}
B_{jk} = \langle Ab_k,Ab_j\rangle,\quad L_{jk} = \langle Ab_k,b_j\rangle,\quad M_{jk} = \langle b_k,b_j\rangle.
\end{equation}
With $Q(z)=B - 2zL + z^2M$, it is readily seen that $z\in\Spec_2(A,\mathcal{L})$ if and only if
\begin{equation}\label{qep}
\exists \underline{u}\in\C^n\setminus\{0\} \quad \text{with}\quad
Q(z)\underline{u} = \underline{0},
\end{equation}
that is, $\Spec_2(A,\mathcal{L}) = \Spec(Q(z))$. Evidently, $\Spec_2(A,\cL)$ is independent of the basis chosen, and since $\det M\not=0$, it follows that $\Spec_2(A,\cL)$ consists of at most $2n$ points. The quadratic eigenvalue problem \eqref{qep} may be solved via suitable linearisations, cf. \cite[Chapter 12]{glr2} for details. For example, if
\begin{equation*}
T =\begin{pmatrix} 0 & I \\ -B & 2L \end{pmatrix}
\quad \text{and}\quad
S =\begin{pmatrix} I & 0\\ 0 & M \end{pmatrix},
\end{equation*}
then $z\in\Spec_2(A,\mathcal{L})$ if and only if
$(T - z S)\underline{v}=0$ for some $\underline{v}\not=\underline{0}$. Equivalently, one can consider the matrix $\mathcal{T}$ where
\begin{equation} \label{linmatrix}
\mathcal{T}:=S^{-1}T =\begin{pmatrix} 0 & I \\ -M^{-1}B & 2M^{-1}L \end{pmatrix},
\end{equation}
clearly $\Spec(\mathcal{T})=\Spec(Q(z))=\Spec_2(A,\mathcal{L})$. Consider also the non-singular matrices
\begin{displaymath}
E(z) :=\begin{pmatrix} zM - 2L & M \\ I & 0 \end{pmatrix}\quad\textrm{and}\quad
F(z) :=\begin{pmatrix} I & 0 \\ zI & I \end{pmatrix}.
\end{displaymath}
A straight forward calculation yields $Q(z)\oplus (-I) = E(z)(zI - \mathcal{T})F(z)$. The matrices $Q(z)\oplus (-I)$ and $(zI - \mathcal{T})$ are said to be equivalent. Also, every matrix polynomial $L(z)$ is equivalent to a diagonal matrix polynomial
\begin{displaymath}
\textrm{diag}\big[i_1(z),i_2(z),\dots,i_r(z),0,\dots,0\big]
\end{displaymath}
where the diagonal entries $i_j(z)$ are polynomials of the form $\Pi_k(z - z_{jk})^{\beta_{jk}}$ with the property that $i_j(z)$ is divisible by $i_{j-1}(z)$. The factors $(z - z_{jk})^{\beta_{jk}}$ are called elementary divisors and the $z_{jk}$ are eigenvalues of $L(z)$. The degrees of the elementary divisors associated to a particular eigenvalue $z_0$ are called the partial multiplicities of $L(z)$ at $z_0$; see \cite[Theorem A.6.1]{glr2} for further details. For a linear matrix polynomial (such as $\mathcal{T} - zI$) the degrees of the elementary divisors associated to a particular eigenvalue $z_0$ coincide with the sizes of the Jordan blocks associated to $z_0$ as an eigenvalue of $\mathcal{T}$; see \cite[Theorem A.6.3]{glr2}. The notion of multiplicity of an eigenvalue now follows from the fact that two $n\times n$ matrix polynomials are equivalent if and only if they have the same collection of elementary divisors; see \cite[Theorem A.6.2]{glr2}.

\begin{definition} \label{multip}
The geometric and algebraic multiplicity of an element $z\in\Spec_2(A,\cL)$ are defined to be the geometric and algebraic multiplicity of $z$ as an eigenvalue of $\mathcal{T}$.
\end{definition}

Note that the definition of geometric and algebraic multiplicity is independent of the basis used to assemble the matrices $B,L,M$ and $\mathcal{T}$. Indeed, let $B,L,M$ be assembled with respect a basis $b_1,\dots,b_n$ and $\tilde{B},\tilde{L},\tilde{M}$ be assembled with respect a basis $c_1,\dots,c_n$, where $\cL=\Span\{b_j\}_{j=1}^n = \Span\{c_j\}_{j=1}^n$. Then,
\[
\tilde{\mathcal{T}}:=\begin{pmatrix} 0 & I \\ -\tilde{M}^{-1}\tilde{B} & 2\tilde{M}^{-1}\tilde{L} \end{pmatrix} = \begin{pmatrix} N^{-1} & 0 \\ 0 & N^{-1} \end{pmatrix}\begin{pmatrix} 0 & I \\ -M^{-1}B & 2M^{-1}L \end{pmatrix}\begin{pmatrix} N & 0 \\ 0 & N \end{pmatrix}
\]
where $N_{ij} = \langle c_j,b^*_i\rangle$, and therefore $\tilde{\mathcal{T}}$ and $\mathcal{T}$ are equivalent.

\subsection{The approximate spectral distance} Suppose that the given basis $\{b_j\}_{j=1}^n$ of $\cL$ is orthonormal so that $M=I$ in \eqref{matrices}. Let
$Q(z)=B - 2zL + z^2I$. For $z\in \C$ we define
\[
   \sigma_{A,\cL}(z)=\min_{\underline{v}\in\C^n}\frac{\|Q(z)\underline{v}\|}{\|\underline{v}\|}.
\]
Note that the right hand side is independent of the orthonormal basis chosen for $\cL$. If $\cL\subset \Dom(A^2)$, then
\[
    \sigma_{A,\cL}(z)=\min_{v\in\cL}\frac{\|P(A-z)^2v\|}{\|v\|},
\]
and for $z\not\in \Spec_2(A,\cL)$ we have $\sigma_{A,\cL}(z)=\|Q(z)^{-1}\|^{-1}$. Furthermore, $z\in \Spec_2(A,\cL)$ if and only if $\sigma_{A,\cL}(z)=0$.
The map $\sigma_{A,\cL}:\C\longrightarrow [0,\infty)$ is a Lipschitz
subharmonic function; see \cite[Lemma~1]{da0}, \cite[Lemma~4.1]{bo2} and \cite[Theorem~2]{blp}. Therefore $\Spec_2(A,\cL)$ is completely characterised via
the maximum principle.

\subsection{The spectrum and the second order spectra}
Let
\[
\lambda_{\min}(A) =
\inf [\Spec(A)] \quad\textrm{and}\quad
\lambda_{\max}(A) =
\sup [\Spec(A)].
\]
For all $\cL\subset \Dom(A)$,
\begin{equation}\label{eugene0}
    \Spec_2(A,\cL)\subset \D[\lambda_{\min}(A),\lambda_{\max} (A)];
\end{equation}
see \cite[Theorem 3.1]{shar} and \cite[Corollary 2.1]{str}. Thus
\[
    \lim_{n\to\infty}\Spec_2(A,\cL_n)\subseteq \D[\lambda_{\min}(A),\lambda_{\max} (A)]
\]
for any sequence $(\cL_n)$.

A growing interest in the second order relative spectrum and corresponding limit sets has been stimulated by the following property: if $(a,b)\cap\Spec(A)=\emptyset$, then
\begin{equation}\label{eugene1}
\Spec_2(A,\mathcal{L})\cap \mathbb{D}(a,b) = \varnothing \qquad \qquad \forall
\cL\subset \Dom(A);
\end{equation}
see \cite[Theorem 5.2]{lesh} or Lemma~\ref{lem1} below.
Thus,
\begin{equation}\label{eugene2}
\Spec(A)\cap\big[\Re z - \vert\Im z\vert,\Re z + \vert\Im
z\vert\big]\ne\varnothing \qquad \text{whenever}\qquad z\in\Spec_2(A,\mathcal{L});
\end{equation}
see \cite[Corollary 4.2]{shar} and \cite[Theorem 2.5]{lesh}.
Thus inclusions of points in the spectrum of $A$ are achieved from $\Re z$ with a two-sided explicit residual given by $|\Im z|$.

In fact, the order of magnitude of the residue in the approximation of $\Spec(A)$ by projecting $\Spec_2(A,\cL)$ into $\R$ can be improved to $|\Im z|^2$, if some information on the localisation of $\Spec(A)$ is at hand. Indeed, if $(a,b)\cap\Spec(A)=\{\lambda\}$ and $z\in \Spec_2(A,\mathcal{L})$ with $z\in \mathbb{D}(a,b)$, then
\begin{equation}\label{strauss}
\bigg{[}\Re z
- \frac{\vert\Im z\vert^2}{b-\Re z},\Re z
+ \frac{\vert\Im z\vert^2}{\Re z - a}\bigg{]}\cap\Spec(A)=\{\lambda\};
\end{equation}
see \cite[Corollary~2.6]{bole} and \cite[Theorem 2.1 and Remark 2.2]{str}.

The following lemma is an improvement upon \cite[Theorem~5.2]{shar}. It immediately implies property \eqref{eugene1}.

\begin{lemma}\label{lem1}
Let $a,b\in \R$ be such that $(a,b)\cap \Spec(A)=\varnothing$. If $z\in \D(a,b)$, then
\begin{equation}\label{lower_bd_sigma}
 \sigma_{A,\cL}(z) \geq \alpha(z)
\end{equation}
where
\begin{equation} \label{alpha}
     \alpha(z) =\alpha_{(a,b)}(z)= \frac{(b-a)^2-|z-a|^2-|z-b|^2}{2|z-b||z-a|}\dist(z,\{a,b\})^2>0.
\end{equation}
\end{lemma}
\begin{proof}
Without loss of generality assume that $\Im z\ge 0$. The region \[F_z :=\{\mu-z:\mu\in\Spec(A)\}\] is contained in two sectors, one between the real line and the ray $re^{-i\theta_1}$ $(r\ge 0)$, the other between the real line and $re^{i(\pi+\theta_2)}$, where $0\leq\theta_1+\theta_2< \pi/2$. We have $\cos(\theta_1+\theta_2)>0$, and by the cosine rule,
\[
     \cos(\theta_1+\theta_2)=\frac{(b-a)^2-|z-b|^2-|z-a|^2}{2|z-b||z-a|}.
\]
Now $(F_z)^2 = \{w^2:w\in F_z\}$ is contained in a sector between the rays
$re^{-i2\theta_1}$ and $re^{i2\theta_2}$. Therefore
\begin{displaymath}
\min\{\Re y:y\in e^{i(\theta_1-\theta_2)}(F_z)^2\} = \min\{| a - z|^2,| b - z|^2\}\cos(\theta_1+\theta_2)
=\alpha(z).
\end{displaymath}
By virtue of the spectral theorem,
\begin{align*}
\Re e^{i(\theta_1-\theta_2)}\langle Q(z)\underline{v},\underline{v}\rangle_{\C^n}&=
\Re e^{i(\theta_1-\theta_2)}\langle (A-z)v,(A-\overline{z})v\rangle \\
&= \Re e^{i(\theta_1-\theta_2)}\int_\mathbb{R}(\mu-z)^2~d\langle E_\mu v,v\rangle\\
&\ge \alpha(z)  \|v\|^2.
\end{align*}
\end{proof}

From \eqref{eugene1} it follows that
\[
    \D(a,b)\cap \lim_{n\to\infty} \Spec_2(A,\cL_n)=\varnothing
    \quad \text{whenever} \quad (a,b)\cap \Spec(A)=\varnothing
\]
and
\[
    \R\cap \lim_{n\to \infty}\Spec_2(A,\cL_n)\subseteq  \Spec(A).
\]
We now consider two examples where
\begin{equation} \label{equality}
   \lim_{n\to \infty}\Spec_2(A,\cL_n)=  \Spec(A).
\end{equation}
The first one has a compact resolvent while the second one has
$-1$ in the essential spectrum.

\begin{example} \label{ex4}
Let $\cH$, $\cL_n$ and $A$ be as Example~\ref{ex1}. Then
\[
    \Spec_2(A,\cL_n)=\{\pm i n,\pm 1,\ldots,\pm (n-1)\}.
\]
\end{example}

\begin{example} \label{ex5}
Let $\cH$, $\cL_n$ and $A$ be as Example~\ref{ex3}. Then
\[
    \Spec_2(A,\cL_n)=\{\alpha_n \pm i \gamma_n,-1,1,\ldots,(n-1)^r\}
\]
where $\alpha_n=n^r\sin(1/n)^2-\cos(1/n)^2$ and
$\gamma_n=(n^r+1)\sin(1/n)\cos(1/n)$.
The geometric multiplicity of the second order spectral point
$-1$ is $n-1$ and its algebraic multiplicity is $2(n-1)$.
As $n\to \infty$, the non-real point
\[
    \alpha_n \pm i \gamma_n\to \left\{
    \begin{array}{ll}
       -1 & 0<r<1 \\
       -1 \pm i & r=1 \\
       \infty & r>1.
    \end{array} \right.
\]
Hence \eqref{equality} holds true for $r\not=1$.
\end{example}

It is naturally expected that if $(\cL_n)$ approximate
very fast a Weyl sequence for $\lambda\in \Spec(A)$, then
\[
    \lambda\in \lim_{n\to \infty} \Spec_2(A,\cL_n).
\]
This intuition can be made rigorous through the following statement which appears to be of little practical use.

\begin{proposition} \label{full_app}
Assume that $\|A\|<\infty$. Let $\lambda\in \Spec(A)$. Let $(\cL_n)\subset \cH$ be such that
$n=\dim \cL_n$. If there exists $x_n\in\cL_n$ such that $\|x_n\|=1$ and
$
   \|(A-\lambda)x_n\|^{\frac 1{2n}}=\varepsilon_n \to 0,
$
then
\[
    \D(\lambda-\delta_n,\lambda+\delta_n)\cap \Spec_2(A,\cL_n)\not=
    \varnothing,  \qquad \delta_n=2(1+\|A\|)^2\varepsilon_n,
\]
for all $n$ large enough.
\end{proposition}
\begin{proof}
Let $\cT_n=\cT$ be as in \eqref{linmatrix} for an orthonormal basis of $\cL_n$.
If
\begin{equation} \label{full_app_hyp}
\|(\cT_n-\lambda )\underline{v}\|_{\C^{2n}}\leq\tilde{\varepsilon}
\|\underline{v}\|_{\C^{2n}},
\end{equation}
then
\begin{equation} \label{full_app_key}
   \D(\lambda-\delta,\lambda+\delta)\cap \Spec_2(A,\cL)\not =\varnothing
   \quad \text{where} \quad \delta=(\tilde{\varepsilon}\|\cT_n-\lambda \|^{2n-1})^{\frac{1}{2n}}.
\end{equation}
Indeed, suppose that the left side set in \eqref{full_app_key} is empty.
Then $(\cT_n-\lambda) \in \C^{2n\times 2n}$ is invertible and $\dist(\lambda,\Spec(\cT_n))>\delta$. Condition \eqref{full_app_hyp} implies $\|(\cT_n-\lambda)^{-1}\|\geq \tilde{\varepsilon}$. But \cite[Lemma~1]{kato}
\[
    \|(\cT_n-\lambda)^{-1}\|\leq \frac{\|\cT_n-\lambda\|^{2n-1}}{|\det(\cT_n-\lambda)|}
    <\frac{\|\cT_n-\lambda\|^{2n-1}}{\delta^{2n}}=\tilde{\varepsilon}^{-1}.
\]
Hence \eqref{full_app_key} follows from the contradiction.

Let $v_n=x_n\oplus \lambda x_n$ so that $\|v_n\|=\sqrt{1+\lambda^2}$. Then
\[
 \frac{\|(\cT_n-\lambda)\underline{v_n}\|_{\C^{2n}}}{\|\underline{v_n}\|_{\C^{2n}}}
 =\frac{\|P_n(A-\lambda)^2x_n\|}{\sqrt{1+\lambda^2}}\leq
 \frac{\|(A-\lambda)\|\ \|(A-\lambda)x_n\|}{\sqrt{1+\lambda^2}}.
\]
Thus \eqref{full_app_key} holds for
\[
    \delta=\frac{\|A-\lambda\|^{\frac 1{2n}}}{(1+\lambda^2)^{\frac 1{4n}}}
    \|\cT_n-\lambda\|^{1-\frac 1{2n}} \varepsilon_n.
\]
The conclusion is a consequence of the identity $\|\cT_n-\lambda\|\leq (1+\|A\|)^2$.
\end{proof}

\subsection{Mapping of second order spectra}

\begin{lemma} \label{l:mt}
For $a,b,c,d\in \R$, let $f(w)=aw+b$, $g(w)=cw+d$ and $F(w)=f(w)g(w)^{-1}$. If $ad\ne cb$ and $g(A)^{-1}\in\mathcal{B}(\cH)$, then
\begin{equation}\label{map}
z\in \Spec_2(A,\cL)    \quad
\iff \quad F(z)\in \Spec_2(F(A),g(A)\cL).
\end{equation}
Moreover, the algebraic and geometric multiplicities of $z$ and $F(z)$ are the same.
\end{lemma}
\begin{proof}
Since $g(A)$ preserves linear independence, a set of vectors $\{b_j\}_{j=1}^n\subset \cL$ is a basis for $\cL$ if and only if the corresponding set of vectors
$\{g(A)b_j\}_{j=1}^n$ is a basis for $g(A)\cL$. Assume that $\{b_j\}$ is a basis for $\cL$ and let $\tilde{b}_j=g(A)b_j$. It is readily seen that $(\tilde{b}_j)^*=g(A)^{-1}b_j^*$. Let $u,v\in \cL$ and $\tilde{u}=g(A)u$ and $\tilde{v}=g(A)v$. Then $\underline{\tilde{u}}=\underline{u}\in\C^n$
for any $u\in \cL$, obtaining the left side from $\{\tilde{b}_j\}$ and the
right side from $\{b_j\}$. Let $B,~L,~M$ be as in \eqref{matrices} for $A$ and $\{b_j\}$. Let $\tilde{B},~\tilde{L},~\tilde{M}$ be as in \eqref{matrices} for $F(A)$ and $\{\tilde{b}_j\}$. Now, let $z\in \Spec_2(A,\cL)$, and note that $g(A)^{-1}\in\mathcal{B}(\cH)$ implies that $g(z)\ne 0$. We have
\begin{align*}
\langle(B - 2zL + z^2M)\underline{u},\underline{v}\rangle &=\langle(A-z)u,(A-\overline{z})v\rangle\\
&=(ad-cb)^{-2}g(z)^2\langle (F(A)-F(z))\tilde{u},(F(A)-F(\overline{z}))\tilde{v}\rangle\\
&=(ad-cb)^{-2}g(z)^2\langle(\tilde{B} - 2F(z)\tilde{L} + F(z)^2\tilde{M})\underline{u},\underline{v}\rangle,
\end{align*}
since $\underline{u},~\underline{v}\in\mathbb{C}^n$ are arbitrary we deduce that
\begin{equation}\label{same}
B - 2zL + z^2M = (ad-cb)^{-2}g(z)^2(\tilde{B} - 2F(z)\tilde{L} + F(z)^2\tilde{M}).
\end{equation}
Thus $F(z)\in \Spec_2(F(A),g(A)\cL)$, and moreover, since the function $(ad-cb)^{-2}g(w)^2$ is analytic and non-zero at $z$, the algebraic and geometric multiplicities of $z$ and $F(z)$ are the same; see \cite[Theorem A.6.6]{glr2}.
\end{proof}

%%%%%%%%%%%%%%%%%%%%

\section{Accumulation of the second order spectrum and multiplicity}
\label{near_real}
\subsection{Neighbourhoods of the discrete spectrum}
Throughout this section we assume that
\begin{equation} \label{hyp1thm4}
  (a,b)\cap\Spec(A) = \{\lambda_1<\dots<\lambda_s\}\subseteq\Spec_{\dis}(A).
\end{equation}
Let
\[
  \chi_j=\chi_{(\lambda_j-\varepsilon,\lambda_j+\varepsilon)}(A)
\]
where $\varepsilon>0$ is small enough to ensure that $m_j:=\operatorname{Rank}(\chi_j)$ is equal to the multiplicity of $\lambda_j$.
Let $m=\sum_{j=1}^s m_j$ be the total-multiplicity of the group of eigenvalues $\lambda_1,\dots,\lambda_s$. We will denote by
$\cE$ the eigenspace associated to the group of eigenvalues in $(a,b)$, that is, $\cE=\operatorname{Range}\chi_{(a,b)}(A)$. We fix an orthonormal basis $\cB=\{u_j\}_{j=1}^m$ of $\cE$, where the $u_j$ are eigenvectors associated to the eigenvalues $\lambda_j$.

\begin{lemma}\label{lem2}
Let $z\in \mathbb{D}(a,b)$. Let
\begin{equation} \label{gamma}
\gamma = 2\sqrt{5}m(b-a)^2\big[(1+\max\{|a|,|b|\})^2+(b-a)^2(2\sqrt{5}s+8)\big].
\end{equation}
If the subspace $\cL\subset \Dom(A^2)$ is such that
\[
\dist_{\Dom(A^2)}(\cB,\cL)<\delta< \alpha_{(a,b)}(z)\big[\dist(z,\Spec\, A)\big]^2\gamma^{-1},
\]
then
\[
      \sigma_{A,\cL}(z)\ge\frac{2}{5|b-a|^2}(\alpha_{(a,b)}(z)\big[\dist(z,\Spec\, A)\big]^2-\delta \gamma)>0.
\]
\end{lemma}
\begin{proof}
Let $c\in\mathbb{R}\backslash(a,b)$. Define
\[
    \tilde{A}=A+\sum_{j=1}^s(c-\lambda_j)\chi_j \qquad
    \Dom(\tilde{A})=\Dom(A).
\]
Then $\tilde{A}$ is self-adjoint and $\Spec(\tilde{A})=\{c\}\cup \Spec(A)\setminus \{\lambda_1,\ldots,\lambda_s\}$. Let
\[
K(z) = - \sum_{j=1}^s(c-\lambda_j)(\lambda_j + c - 2z)\chi_j,
\]
then $K(z)$ is a finite rank operator with range $\cE$ and
\[
(\tilde{A}-z)^2 =
(A-zI)^2-K(z).
\]
Evidently, $(\tilde{A}-z)^2$ is invertible and for all $v\in\mathcal{H}$ we have
\begin{align*}
\Vert v + (\tilde{A}-z)^{-2}&K(z)v\Vert^2\\ &= \Vert \chi_{\mathbb{R}\backslash(a,b)}v\Vert^2 + \Vert \chi_{(a,b)}v + (\tilde{A}-z)^{-2}\big((A-z)^2 - (\tilde{A}-z)^2\big)\chi_{(a,b)}v\Vert^2\\
&= \Vert \chi_{\mathbb{R}\backslash(a,b)}v\Vert^2 + \Vert (\tilde{A}-z)^{-2}(A-z)^2\chi_{(a,b)}v\Vert^2\\
&= \Vert \chi_{\mathbb{R}\backslash(a,b)}v\Vert^2 + \sum_{i=1}^s\vert\lambda_i - z\vert^4\Vert \chi_iv\Vert^2/\vert c-z\vert^4\\
&\ge \beta_1(z)^4\|v\|^2,
\end{align*}
where $\beta_1(z)=\min\{1,\dist(z,\Spec(A))/|c-z|\}$.

Let $Q(z) = P(\tilde{A}-z)^2\upharpoonright_{\mathcal{L}}$. For each eigenvector $u_j$ we assign a $v_j\in\mathcal{L}$ such that $\Vert A^p(u_j-v_j)\Vert\le\delta$ for $p=0,1,2$. Note that in general $Pu_j\not =v_j$, however $\|Pu_j-v_j\|\leq \delta$. The hypothesis on $\delta$ ensures that
\begin{align*}
\Vert Q(z)v_j-(\tilde{A}-z)^2 u_j\Vert &\le \Vert P(\tilde{A}-z)^2(u_j-v_j)\Vert + \vert c-z\vert^2\Vert Pu_j-u_j\Vert\\
&\le \Vert (\tilde{A}-z)^2(u_j-v_j)\Vert + \vert c-z\vert^2\Vert Pu_j-u_j\Vert\\
&\le \Vert [(A-zI)^2 + \sum_{i=1}^s(c-\lambda_i)(\lambda_i + c - 2z)\chi_i](u_j-v_j)\Vert + \delta\vert c-z\vert^2\\
&\le \delta \beta_2(z),
\end{align*}
where $\beta_2(z)=(1+|z|)^2+s \mu(z)+|c-z|^2$ and $\mu(z)=\max_{1\leq j\leq s}\vert(c-\lambda_j)(\lambda_j + c - 2z)\vert$.

Note that $\tilde{A}$ satisfies the hypothesis of Lemma~\ref{lem1}.
Let $\alpha(z)$ be given by \eqref{alpha}. Then
\begin{align*}
\Vert (\tilde{A}-z)^{-2}u_j&-Q(z)^{-1}Pu_j\Vert \\
&\le\Vert (\tilde{A}-z)^{-2}u_j-Q(z)^{-1}v_j\Vert +\Vert Q(z)^{-1}Pu_j-Q(z)^{-1}v_j\Vert\\
&\le\Vert (c-z)^{-2}u_j-Q(z)^{-1}v_j\Vert + \alpha(z)^{-1}\delta\\
&\le\Vert (c-z)^{-2}v_j-Q(z)^{-1}v_j\Vert + (\alpha(z)^{-1}+\vert c-z\vert^{-2})\delta\\
&\le \alpha(z)^{-1}\Vert (c-z)^{-2}Q(z)v_j-v_j\Vert + (\alpha(z)^{-1}+\vert c-z\vert^{-2})\delta\\
&\le \alpha(z)^{-1}\Vert (c-z)^{-2}Q(z)v_j-u_j\Vert + (2\alpha(z)^{-1}+\vert c-z\vert^{-2})\delta\\
&\le \alpha(z)^{-1}\vert c-z\vert^{-2}\Vert Q(z)v_j-(\tilde{A}-z)^2u_j\Vert + (2\alpha(z)^{-1}+\vert c-z\vert^{-2})\delta \\
&\leq \beta_3(z) \delta
\end{align*}
where $\beta_3(z)=\alpha(z)^{-1}|c-z|^{-2}[\beta_2(z)+2|c-z|^2+\alpha(z)]$.

Thus
\begin{align*}
\Vert[(\tilde{A}-z)^{-2}&-Q(z)^{-1}P]K(z)v\Vert\\
&= \Big\Vert[(\tilde{A}-z)^{-2}-Q(z)^{-1}P]\sum_{j=1}^s(c-\lambda_j)(2z - c- \lambda_j)\chi_jv\Big\Vert\\
&\le\sum_{j=1}^s|c-\lambda_j|\vert 2z - c- \lambda_j\vert\Vert[(\tilde{A}-z)^{-2}-Q(z)^{-1}P]\chi_jv\Vert\\
&\le \mu(z) \|v\|\sum_{k=1}^m\Vert[(\tilde{A}-z)^{-2}-Q(z)^{-1}P]u_k\Vert\\
&\le \delta m \mu(z) \beta_3(z) \Vert v\Vert.
\end{align*}

Hence, for any $v\in\cL$,
\begin{align*}
\Vert P(A&-zI)^2v\Vert = \Vert Q(z)v + PK(z)v\Vert\\
&\ge \|Q(z)^{-1}\|^{-1}\Vert v + Q(z)^{-1}PK(z)v\Vert\\
&\ge \alpha(z)[\Vert v + (\tilde{A}-z)^{-2}K(z)v\Vert-
\Vert (\tilde{A}-z)^{-1}K(z)v-Q(z)^{-1}PK(z)v\Vert]\\
&\ge \alpha(z)[\beta_1(z)^2-\delta m \mu(z) \beta_3(z)] \|v\|.
\end{align*}
By fixing
\[
     c=\left\{ \begin{array}{lll} b+(b-a)=2b-a & \text{if} & \Re z\geq (a+b)/2 \\
     a-(b-a)=2a-b & \text{if} & \Re z \leq (a+b)/2, \end{array}\right.
\]
we get
\[
     b-a\leq |c-z| \leq \frac{\sqrt{10}}{2} (b-a).
\]
Thus
\begin{gather*}
\beta_1(z) = \frac{\dist(z,\Spec(A))}{|c-z|},\qquad \mu(z) \leq 2\sqrt{5}(b-a)^2,\\
\beta_2(z)\leq(1+\max\{|a|,|b|\})^2+(b-a)^2(2s\sqrt{5}+10/4)\qquad\text{and}
\qquad \alpha(z)\le\frac{(b-a)^2}{2}.
\end{gather*}
Therefore
\begin{align*}
\sigma_{A,\cL}(z) &\ge \alpha(z)[\beta_1(z)^2-\delta m \mu(z) \beta_3(z)]\\
&\ge \frac{\alpha(z)[\dist(z,\Spec(A))]^2}{|c-z|^2}-\frac{\delta m 2\sqrt{5}(b-a)^2[\beta_2(z)+2|c-z|^2+\alpha(z)]}{|c-z|^2}\\
&\ge \frac{\alpha(z)[\dist(z,\Spec(A))]^2-\delta \gamma}{|c-z|^2}\\
&\ge \frac{2}{5|b-a|^2}(\alpha(z)[\dist(z,\Spec(A))]^2-\delta \gamma).
\end{align*}
\end{proof}

\begin{lemma}\label{multiplicity}
Let $\lambda\in\Spec_{\dis}(A)$ be an eigenvalue of multiplicity $m$ and let
\[d_\lambda=2^{-1/2}\dist(\lambda,\Spec(A)\backslash\{\lambda\}).\] If $\mathcal{E}\subset\mathcal{L}$, then
\begin{equation}\label{isolation}
\Spec_2(A,\mathcal{L})\cap\mathbb{D}(\lambda-d_\lambda,\lambda+d_\lambda) = \{\lambda\} \qquad
\end{equation}
and $\lambda$ as member of $\Spec_2(A,\cL)$ has geometric multiplicity $m$ and algebraic multiplicity $2m$.
\end{lemma}
\begin{proof}
Assume that $z\in\Spec_2(A,\mathcal{L})\cap\mathbb{D}(\lambda-d_\lambda,\lambda+d_\lambda)$ and $z\ne\lambda$. Then there exists a normalised $u\in\mathcal{L}\backslash\{0\}$ such that
\begin{equation} \label{evector}
\langle(A-zI)u,(A-\overline{z}I)w\rangle = 0
\end{equation}
for all $w\in \cL$. In particular $\langle u,u_k\rangle = 0$ for each $u_k\in\cE$.
If we take $w=u$ in \eqref{evector}, we achieve
\[
    \|(A-\Re z)u\|^2-|\Im z|^2 \|u\|^2-2 i |\Im z| \langle (A-\Re z) u,u\rangle=0.
\]
Thus $|\Im z|=\|(A-\Re z)u\|\not=0$ and $\langle (A-\Re z)u,u\rangle =0$, so that
\begin{align*}
    \|(A-\lambda) u\|^2&=|\lambda-\Re z|^2\|u\|^2+
    \|(A-\Re z)u\|^2+2 (\lambda-\Re z)\langle (A-\Re z)u,u\rangle \\
    &   = |\lambda - \Re z|^2+|\Im z|^2\leq 2 |\Im z|^2.
\end{align*}
Since
\[
\dist(\lambda,\Spec(A)\setminus \{\lambda\})\leq \frac{\|(A-\lambda) (I-
\chi_{(\lambda-d_\lambda,\lambda+d_\lambda)})u\|}{\|(I-
\chi_{(\lambda-d_\lambda,\lambda+d_\lambda)})u\|}=\|(A-\lambda)u\|,
\]
we get
\[
     \|u\|=\|(I-
\chi_{(\lambda-\epsilon,\lambda+\epsilon)})u\|\leq
\frac{\sqrt{2}|\Im z|}{\dist(\lambda,\Spec(A)\setminus \{\lambda\})}<1.
\]
Statement \eqref{isolation} follows from the contradiction.

Let $\{b_1,\dots,b_n\}$ be a basis for $\mathcal{L}$ and let $B$, $L$ and $M$ be the matrices \eqref{matrices} associated to this basis. Let $v$ be an arbitrary member of $\mathcal{L}$. We have
\[
0 = \langle (A - \lambda I)u_k,(A - \lambda I)v)\rangle_{\mathcal{H}}= \langle(B - 2\lambda L + \lambda^2M)\underline{u_k},\underline{v}\rangle_{\mathbb{C}^n}
\]
so that $(B - 2\lambda L + \lambda^2M)\underline{u_k} = 0$. Hence
\begin{displaymath}
\Big[T - \lambda S\Big]\left(
\begin{array}{c}
\underline{u_k}\\
\lambda\underline{u_k}
\end{array} \right) = \Bigg[\left(
\begin{array}{cc}
0 & I\\
-B & 2L
\end{array} \right)-\lambda \left(
\begin{array}{cc}
I & 0\\
0 & M
\end{array} \right)\Bigg] \left(
\begin{array}{c}
\underline{u_k}\\
\lambda\underline{u_k}
\end{array} \right) = 0.
\end{displaymath}
Since $\det(S)\not=0$, $\left\{\binom{\underline{u_1}}{\lambda\underline{u_1}},\dots,\binom{\underline{u_m}}{\lambda\underline{u_m}}\right\}$ are eigenvectors of $\cT$ corresponding to the eigenvalue $\lambda$. For $k\ne j$ we have $\langle M\underline{u_k},\underline{u_j}\rangle_{\mathbb{C}^n} = \langle u_k,u_j\rangle_{\mathcal{H}} = 0$, so the vectors $\{\underline{u_1},\dots,\underline{u_m}\}$ are linearly independent. There is a one to one correspondence between the eigenvectors of $\cT$ associated to $\lambda$ and the eigenvectors of $A$ associated to $\lambda$. Thus, the geometric multiplicity of $\lambda\in\Spec_2(A,\cL)$ is equal to $m$.

Now let us compute the algebraic multiplicity. Let $\underline{v},\underline{w}\in\mathbb{C}^n$ and $\underline{u}\in\operatorname{Span}\{\underline{u_k}\}_{k=1}^m$. If
\begin{align*}
\left(
\begin{array}{c}
\underline{u}\\
\lambda\underline{u}
\end{array} \right) &= [\cT - \lambda I]\left(
\begin{array}{c}
\underline{v}\\
\underline{w}
\end{array} \right)\\
&= \Bigg[\left(
\begin{array}{cc}
0 & I\\
-M^{-1}B & 2M^{-1}L
\end{array} \right)-\lambda \left(
\begin{array}{cc}
I & 0\\
0 & I
\end{array} \right)\Bigg] \left(
\begin{array}{c}
\underline{v}\\
\underline{w}
\end{array} \right),
\end{align*}
we have $\underline{w} = \lambda\underline{v} + \underline{u}$ and $-M^{-1}B\underline{v} + 2M^{-1}L\underline{w} - \lambda\underline{w} = \lambda\underline{u}$. Therefore
\begin{displaymath}
-B\underline{v} + 2\lambda L\underline{v} - \lambda^2M\underline{v} = 2\lambda M\underline{u} - 2L\underline{u} = 0
\end{displaymath}
so we deduce that
\begin{displaymath}
\left(
\begin{array}{c}
\underline{v}\\
\underline{w}
\end{array} \right) = \left(
\begin{array}{c}
\underline{v}\\
\lambda\underline{v}
\end{array} \right) + \left(
\begin{array}{c}
\underline{0}\\
\underline{u}
\end{array} \right)\quad\textrm{where}\quad\underline{v}\in
\operatorname{Span}\{\underline{u_k}\}_{k=1}^m.
\end{displaymath}
Suppose now that
\begin{align*}
\left(
\begin{array}{c}
\underline{0}\\
\underline{u}
\end{array} \right) = [\cT - \lambda I]\left(
\begin{array}{c}
\underline{v}\\
\underline{w}
\end{array} \right).
\end{align*}
We have $\underline{w} = \lambda\underline{v}$ and $-M^{-1}B\underline{v} + 2M^{-1}L\underline{w} - \lambda\underline{w} = \underline{u}$, therefore $-B\underline{v} + 2\lambda L\underline{v} - \lambda^2M\underline{v} = M\underline{u}$. Thus
\begin{displaymath}
\langle M\underline{u},\underline{u}\rangle = -\langle B\underline{v} - 2\lambda L\underline{v} + \lambda^2M\underline{v},\underline{u}\rangle = -\langle \underline{v},B\underline{u} - 2\lambda L\underline{u} + \lambda^2M\underline{u}\rangle = 0
\end{displaymath}
so that $\underline{u} = \underline{0}$. This ensures that spectral subspace associated with $\lambda$ as eigenvalue of $\cT$ is given by
\begin{displaymath}
\operatorname{Span}\left\{
\left(\begin{array}{c}
\underline{u_1}\\
\lambda\underline{u_1}
\end{array} \right),
\left(\begin{array}{c}
\underline{0}\\
\underline{u_1}
\end{array} \right),\dots,
\left(\begin{array}{c}
\underline{u_m}\\
\lambda\underline{u_m}
\end{array} \right),
\left(\begin{array}{c}
\underline{0}\\
\underline{u_m}
\end{array} \right)\right\}.
\end{displaymath}
\end{proof}

\begin{lemma}\label{aux}
Let $\mathcal{K}$ be a Hilbert space and $\mathcal{A} = \{u_1,\dots,u_m\}\subset\mathcal{K}$ be a set of orthogonal vectors with $\Vert u_j\Vert_\cK\ge 1$ for $1\le j\le m$. Let $\mathcal{L}$ be an $n$-dimensional subspace of $\mathcal{K}$ where $n\ge m$. Let $w_j=\tilde{P}u_j$ where $\tilde{P}:\cK\longrightarrow \cL$
is the orthogonal projection associated to $\cL$.  Let $w_j(t) = tw_j + (1 - t)u_j$
for $t\in[0,1]$. If $\Vert w_j-u_j\Vert< 1/\sqrt{m}$ for each $1\le j\le m$, then there exist vectors $\{w_{m+1},\dots,w_n\}$ such that $\{w_1(t),\dots,w_m(t),w_{m+1},\dots,w_n\}$ is linearly independent for all $t\in[0,1]$.
\end{lemma}
\begin{proof}
If $a_1w_1 + \dots + a_mw_m = 0$ where not all of the $a_j=0$, we would have
\begin{align*}
\sum_{j=1}^m\vert a_j\vert^2 &= \Vert a_1(u_1 - w_1) + \dots + a_m(u_m - w_m)\Vert^2\\
&\le \bigg(\sum_{j=1}^m\vert a_j\vert\Vert u_j - w_j\Vert\bigg)^2 < \bigg(\sum_{j=1}^m\vert a_j\vert\bigg)^2/m
\end{align*}
which contradicts H\"older inequality. Hence necessarily $\{w_1,\ldots,w_m\}$
is linearly independent.
Let $\{w_{m+1},\dots,w_n\}$ be any completion of $\{w_1,\dots,w_m\}$ to a basis of $\mathcal{L}$. Let $t\in[0,1]$. Suppose now that $a_1w_1(t) + \dots + a_mw_m(t) + a_{m+1}w_{m+1} + \dots + a_nw_n = 0$. Then
\begin{align*}
0 &= \sum_{j=1}^m a_jw_j(t) + \sum_{j=m+1}^na_jw_j = \sum_{j=1}^na_jw_j + (I - \tilde{P})\sum_{j=1}^{m}(1-t)a_ju_j.
\end{align*}
The two terms on the right-hand side are orthogonal and therefore each must vanish. As the set $\{w_1,\dots,w_n\}$ is linearly independent, all $a_j=0$ ensuring the conclusion of the lemma.
\end{proof}

%%%%%%%%%%

\subsection{Main result}

\begin{theorem}\label{thm4}
Let $a,b\not\in \Spec(A)$ be such that condition \eqref{hyp1thm4} hold for a group of eigenvalues $\{\lambda_1<\ldots<\lambda_s\}$ with corresponding multiplicities $m_j$. Let $m=\sum_{j=1}^sm_j$ be their total multiplicity. Let
\[
    d=\dist(\{a,b\},\Spec(A)\setminus \{\lambda_j\}_{j=1}^s).
\]
Let $\kappa =d^2/\gamma$ where $\gamma$ is defined by \eqref{gamma}.
Let
\begin{displaymath}
0<\varepsilon < \min \left\{\frac{1}{m^{1/4}\kappa^{1/2}},\min_{0\le j\le s}\frac{|\lambda_j-\lambda_{j+1}|}{2}\right\}
\qquad \text{where} \quad \lambda_0=a\ \text{and}\ \lambda_{s+1}=b.
\end{displaymath}
If $\cL\subset \Dom(A^2)$ is such that
\[
\dist_{\Dom(A^2)}(\cB,\cL)<\kappa \varepsilon ^2
\]
for an orthonormal set of eigenfunctions $\cB$ associated to $\{\lambda_j\}_{j=1}^s$, then
\begin{equation}\label{ths1thm4}
     \Spec_2(A,\cL) \cap \D(a,b) \subset \bigcup_{j=1}^s
     \D(\lambda_j-\varepsilon,\lambda_j+\varepsilon)
\end{equation}
and the total algebraic multiplicity of the points in $\Spec_2(A,\cL)\cap \D(\lambda_j-\varepsilon,\lambda_j+\varepsilon)$
is $2m_j$.
\end{theorem}
\begin{proof}
Let $\tilde{a}=a-d$ and $\tilde{b}=b+d$, then
\[
   \inf_{z\in \D(a,b)} \alpha_{(\tilde{a},\tilde{b})}(z) =
   \min_{\theta\in(-\pi,\pi]} \left\{\alpha_{(\tilde{a},\tilde{b})}(z):z=\frac{b-a}{2}e^{i\theta}+
   \frac{a+b}{2}\right\}
\]
To find the right hand side we may assume without loss of generality that $a = -r$ and $b = r$, then for $z = x+iy$ where $x^2+y^2=r$, we have
\[
\alpha_{(\tilde{a},\tilde{b})}(z) =\frac{(2r + 2d)^2-|x+iy+r+d|^2-|x+iy-r-d|^2}{2|x+iy+r+d||x+iy-r-d|}\dist(x+iy,\{-r-d,r+d\})^2.
\]
We assume that $x\ge 0$ as the case where $x< 0$ can be treated analogously. A straightforward calculation shows that
\[
\min_{\scriptsize \begin{array}{c}0\!\leq\! x\! \leq\! r \\ y^2\!+\!x^2\!=r \end{array}}\alpha_{(\tilde{a},\tilde{b})}(x+iy)=d(d+2r)\min_{\scriptsize \begin{array}{c}0\!\leq\! x\! \leq\! r \\ y^2\!+\!x^2\!=r \end{array}}\frac{|x+iy-r-d|}{|x+iy+r+d|} = d^2,
\]
hence Lemma~\ref{lem2} applied to the interval $(\tilde{a},\tilde{b})$
ensures \eqref{ths1thm4}.

Let us prove the second part. In the notation of Lemma~\ref{aux}, let $\mathcal{A}=\cB$, $\mathcal{K}=\Dom(A^2)$ and define subspaces
$\mathcal{L}_t = \Span\{w_1(t),\dots,w_m(t),w_{m+1},\dots,w_n\}$ for $t\in [0,1]$.
Note that $\cB$ is also an orthogonal set in $\Dom(A^2)$ and $\|u_j\|_{\Dom(A^2)}\geq 1$.
Since $\kappa\varepsilon^2<m^{-1/2}$, we get $\dim \cL_t=n$ for all $t\in[0,1]$. Now,
\begin{align*}
\Vert u_j - w_j(t)\Vert_{\Dom(A^2)} &= t\|u_j-w_j\|_{\Dom(A^2)}=t\|(1-\tilde{P})u_j\|_{\Dom(A^2)}\\
&\le \dist_{\Dom(A^2)}(u_j,\mathcal{L}) \quad\textrm{for all}\quad t\in[0,1],
\end{align*}
thus $\dist_{\Dom(A^2)}(\cB,\cL_t)<\kappa \varepsilon ^2$, and by virtue of Lemma~\ref{lem2} we obtain
\begin{displaymath}
\Spec_2(A,\cL_t) \cap \D(a,b) \subset \bigcup_{j=1}^s\D(\lambda_j-\varepsilon,\lambda_j+\varepsilon)\quad\textrm{for all}\quad t\in[0,1].
\end{displaymath}
Denote by $\cT(t)$ the linearisation matrix defined
as in \eqref{linmatrix} for $\cL_t$. According to Lemma~\ref{multiplicity}, the algebraic multiplicities of $\lambda_j\in\Spec\, \cT(1)$ is $2m_j$.
If we let  $\pi_j(t)$ be the spectral projection associated to
\[
    \Spec\, \cT(t)\cap \D(\lambda_j-\varepsilon,\lambda_j+\varepsilon),
\]
then
\[
\pi_j(t_1) - \pi_j(t_2)= -\frac{1}{2i\pi}\oint_{\vert\lambda_j-z\vert=\varepsilon}(\cT(t_1)-\zeta)^{-1} - (\cT(t_2)-\zeta )^{-1}\operatorname{d}\zeta.
\]
Thus $\pi_j(t)\to \pi_j(1)$ as $t\to 1$. By virtue of \cite[Lemma 4.10, Section 1.4.6]{katopert}, we have $\rank(\pi_j(0)) = \rank(\pi_j(1))$.
\end{proof}

An immediate consequence of Theorem~\ref{thm4} is that if a sequence of test subspaces
$\cL_n$ approximates a normalised eigenfunction $u$ associated to a simple eigenvalue $\lambda\in \Spec(A)$ at a given rate in the graph norm of $A^2$,
\[
      \|u-u_n\|_{\Dom(A^2)}=\delta(n)\to 0, \qquad u_n\in \cL_n,
\]
then there exist a conjugate pair $z_n,\overline{z_n}\in \Spec_2(A,\cL_n)$
such that $|\lambda-z_n|=O(\delta^{1/2})$. A simple example shows that this estimate is not necessarily optimal.

\begin{example} \label{nonoptimal} Let $\cH=\Span\{e_n\}_{n=0}^\infty$,
$
    A=\sum n|e_n^+\rangle \langle e_n^+|
$
and \[\cL_n=\Span\{e_1,\ldots,e_{n-1},\alpha_n e_0+\beta_n e_n\}\]
where $\alpha_n^2+\beta_n^2=1$ and $\beta_n\to 0$. Then
\[
    \Spec_2(A,\cL_n)=\{1,\ldots,n-1,\gamma_n,\overline{\gamma_n}\}
\]
where $\gamma_n\sim i n \beta_n$. On the other hand
\[
   \|A^p((\alpha_n-1)e_0+\beta_n e_n)\|\sim n^p \beta_n \qquad p=0,1,2.
\]
Thus $|\gamma_n|\sim n\beta_n$ while $\dist_{\Dom(A^2)}(\cB,\cL_n)\sim n^2\beta_n$.
Moreover, note that in general $\dist_{\Dom(A^2)}(\cB,\cL_n)$ might diverge as $n\to \infty$
and still we might be able to recover  $|\gamma_n|\to 0$; take for instance $\beta_n=1/n^{3/2}$.
\end{example}

An application of  Lemma~\ref{l:mt} yields convergence to eigenvalues
under the weaker assumption $\cL_n\in\Lambda_1$.

\begin{corollary}\label{thm1}
Let $-\infty\leq a<b\leq \infty$ be such that $(a,b)\cap\Spec_{\ess}(A) = \emptyset$. If $(\mathcal{L}_n)\in\Lambda_1$, then
\begin{equation}\label{limi}
\lim_{n\to\infty}\Spec_2(A,\mathcal{L}_n)\cap \mathbb{D}(a,b) = \Spec_{\dis}(A)\cap (a,b).
\end{equation}
\end{corollary}
\begin{proof}
Choose $\varepsilon>0$ such that $a+\varepsilon< \min\{\lambda\in\Spec(A)\cap (a,b)\}$ and let $\mu\in(a,a+\varepsilon)$. With $g(w)=w-\mu$, the operator $g(A)^{-1}$ is bounded and $g(w)$ satisfies the hypothesis of Lemma~\ref{l:mt}. Let $\cM_n=g(A)\cL_n$. Then $(\cM_n)\in \Lambda_0$. Indeed, any $u\in \cH$ can be expressed as $u=g(A)v$ for $v\in \Dom(A)$, and since $(\cL_n)\in\Lambda_1$, we can find $v_n\in \cL_n$ such that $g(A)v_n\to g(A)v$. By virtue of Theorem~\ref{thm4} applied to $g(A)^{-1}$,
\[
   \lim_{n\to \infty} \Spec_2(g(A)^{-1},\cM_n)\cap \D(g(b)^{-1},g(a+\varepsilon)^{-1})=
   \Spec_\dis(g(A)^{-1})\cap (g(b)^{-1},g(a+\varepsilon)^{-1}).
\]
The fact that $g(w)$ is a conformal mapping, Lemma~\ref{l:mt} and the spectral mapping theorem ensure the desired conclusion.
\end{proof}

%%%

\begin{example} \label{ex6}
Let $A$ be the operator of examples~\ref{ex3} and \ref{ex5}. By virtue of
\eqref{eugene1} and Corollary~\ref{thm1},
\[
     \lim_{n\to\infty} \Spec_2(A,\cL_n)\subset \Spec(A)\cup (-1+i \R)
\]
for any $(\cL_n)\in \Lambda_1$.
\end{example}

\begin{corollary}\label{cor1}
If $A$ has compact resolvent and $(\mathcal{L}_n)\in\Lambda_1$, then
\begin{displaymath}
\lim_{n\to\infty}\Spec_2(A,\mathcal{L}_n) = \Spec(A)
\end{displaymath}
\end{corollary}

When $A$ is not semi-bounded, this statement is in stark contrast to Example~\ref{com}.

%%%%%%%%%%%%%

%%%%%%%%%%%%%%%%%%%%%%%%%%%%%%%%%%%%%%%%%%%%%%%%%%%%%%%%%%%

\section{Numerical applications} \label{applica}

Calculating second order spectra leads to enclosures for
discrete points of $\Spec(A)$. This can be achieved by combining
\eqref{eugene2} with Theorem~\ref{thm4}. The latter yields
\emph{a priori} upper bounds for the length of the enclosure in terms of bounds for the distance from the test space, $\cL$, to an orthonormal basis of eigenfunctions
$\cB$. In practice, these upper bounds are found from interpolation
estimates for bases of $\cL$.

\subsection{On multiplicity and approximation}
Let $\lambda\in \Spec(A)$ be an isolated eigenvalue of multiplicity $m$, so that \eqref{ths1thm4} ensures an upper bound on the estimation of $\lambda$ by $m$ conjugate pairs of $\Spec_2(A,\cL)$. If $\cL$ approximates well a basis of only $l<m$ eigenfunctions and rather poorly the remaining $m-l$ elements of $\cB$,
then it would be expected that only $l$ conjugate pairs of $\Spec_2(A,\cL)$
will be close to $\lambda$ while the remaining $m-l$ will lie at a substantial distance from this eigenvalue. We illustrates
this locking effect on the multiplicity in a simple numerical experiment.

Let $A=-\Delta=-\partial^2_x-\partial^2_y$ subject to Dirichlet boundary conditions on $\Omega=[0,\pi]^2\subset \R^2$. Then
\[
  \Spec(A)=\{j^2+k^2:j,k\in \N\}
\]
and a family of eigenfunctions is given by $u_{jk}(x,y)=\sin(jx)\sin(ky)$.
Some eigenvalues of $A$ are simple and some are multiple. In particular
the ground eigenvalue $2=1+1$ is simple and $1+k^2$ for $k=2,3,4,5$ are
double. The contrast between the two eigenfunctions $u_{k1}$ and $u_{1k}$
will increase as $k$ increases. The former will be highly oscillatory in the $x$-direction while the later will be so in the $y$-direction. If $\cL$ captures well oscillations in only one direction, we should expect one conjugate pair of $\Spec_2
(A,\cL)$ to be close to $1+k^2$ and the other to be not so close to this eigenvalue.

In order to implement a finite element scheme for the computation of the second order spectra of the Dirichlet Laplacian, the condition $\cL\subset \Dom(A)$ prescribes the corresponding basis to be at least $C^1$-conforming. We let $\cL=\cL(s)$ be generated by a basis of Argyris elements on given triangulations of $\Omega$.
Contrast between the residues in the interpolation of $u_{1k}$ and $u_{k1}$ is achieved by considering triangulation that are stretched either in the $x$ or in the $y$ direction. See Figure~\ref{laplace} right.

\begin{figure}[ht!]
\hspace{-2cm}
\begin{minipage}{5cm}\includegraphics[height=5.5cm]{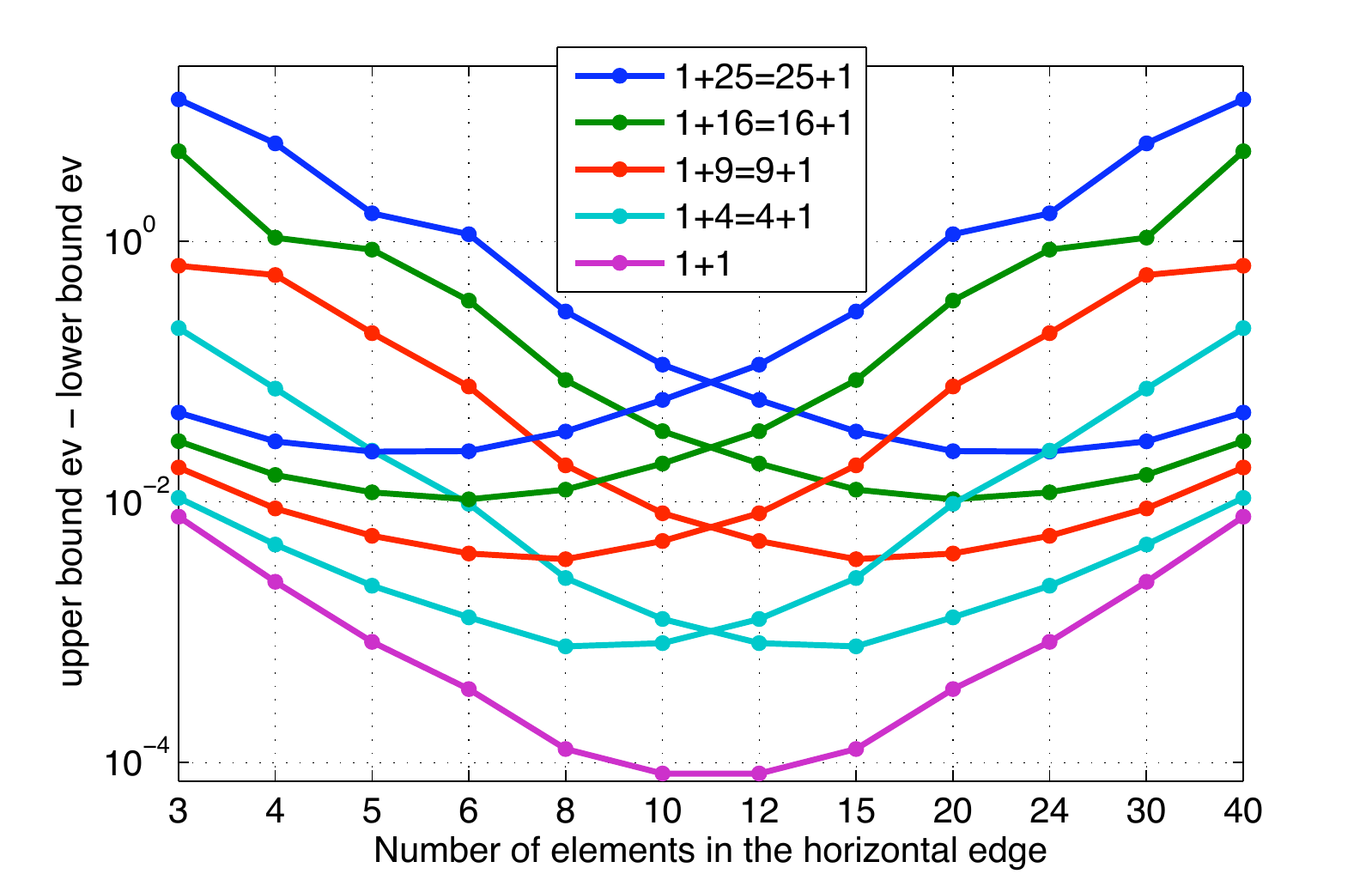}\end{minipage}
\hspace{3cm}
\begin{minipage}{5cm}\includegraphics[height=5.5cm]{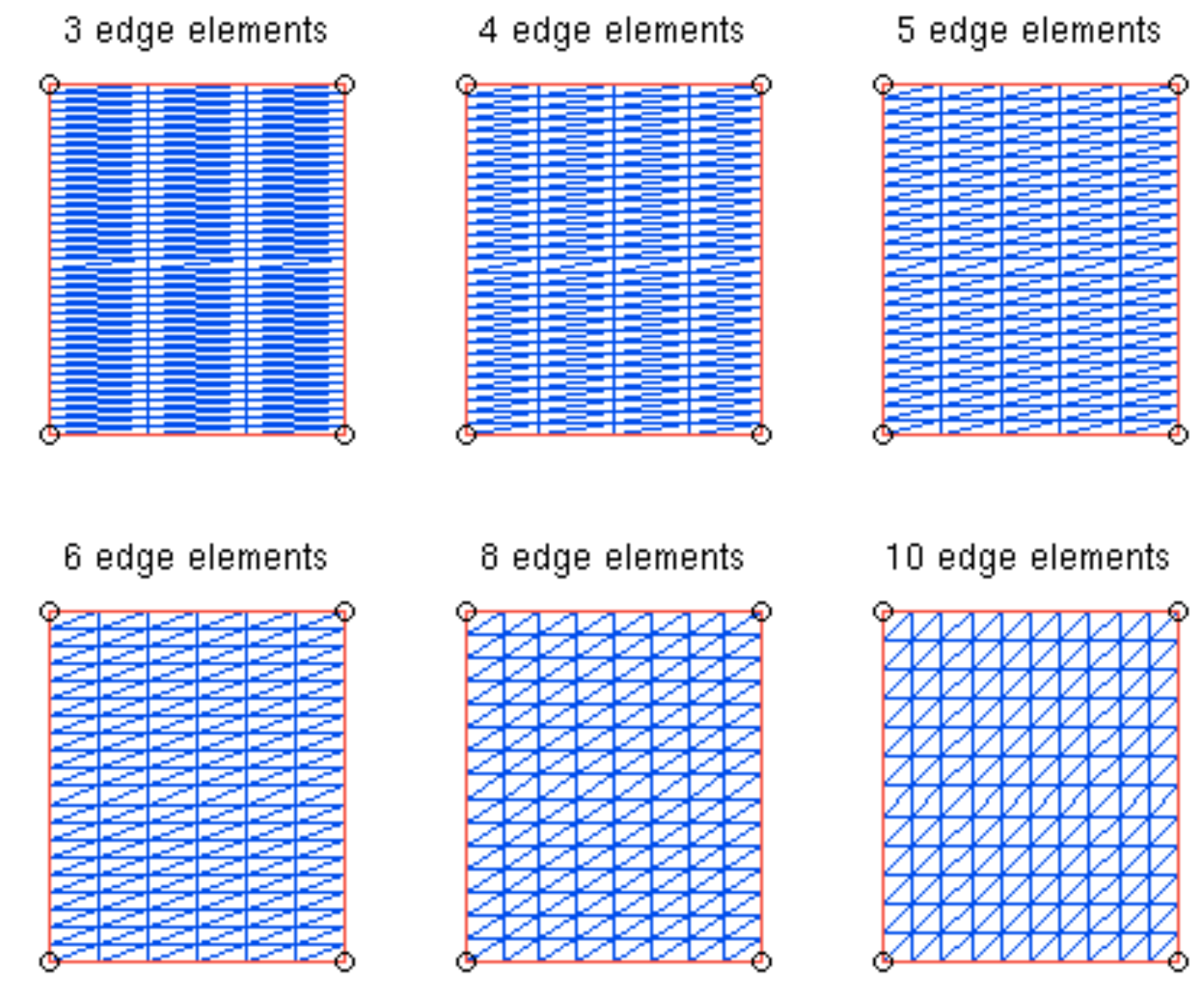}\end{minipage}
\caption{Left: change in the length of the enclosure predicted by \eqref{eugene2} for the eigenvalues $\lambda=1+k^2$ of the 2-dimensional Dirichlet Laplacian on $[0,\pi]^2$. The test subspaces $\cL$ are generated by Argyris elements on a uniform mesh of 240 triangles with a prescribed numbers of segments intersecting the boundary. Right: Corresponding mesh for the first 6 iterates of the left graph.}
\label{laplace}
\end{figure}

In this experiment we generate triangulations with a fixed total number of 240 elements, resulting from diagonally bisecting a decomposition of $\Omega$ as the union of 120 rectangles of equal ratio. We consider mesh with $s=3$, 4, 5, 6, 8, 10, 12, 15, 20, 24, 30 and 40 elements  of equal size on the lower edge $[0,\pi]\times\{0\}\subset \partial \Omega$. We then compute the pairs $z_{1k}(s),\overline{z_{1k}(s)},z_{k1}(s),\overline{z_{k1}(s)}\in\Spec_2(A,\cL(s))$ which are closer to the eigenvalues
$1+k^2$ than to any other point in $\Spec(A)$. By virtue of \eqref{eugene2}, we know that $\Re z_{1k}(s)-|\Im z_{1k}(s)|$ and $\Re z_{k1}(s)-|\Im z_{1k}(s)|$ are lower bounds for this eigenvalue and $\Re z_{1k}(s)+|\Im z_{1k}(s)|$ and $\Re z_{k1}(s)+|\Im z_{1k}(s)|$ are corresponding upper bounds.

\begin{remark} \label{2dlaplacian}
Even though we have the same number of elements in each of the above mesh, in general they do not have the same amount of elements intersecting $\partial \Omega$. Then typically $\dim \cL(s)\not=\dim \cL(t)$ for $s\not=t$, although these numbers do not differ substantially. The precise dimension of the test spaces
are as follows:

\smallskip

\centerline{\begin{tabular}{c|c|c|c|c|c|c}
$s$  & $3$ & $4$ & $5$& $6$ & $8$& $10$ \\
\hline $\dim \cL(s)$ & $2774$   & $2648$  & $2578$  & $2536$   & $2494$  & $2480$
\end{tabular}}
\end{remark}

On the left side of Figure~\ref{laplace} we have depicted the residues $2|\Im z_{1k}(s)|$ and  $2|\Im z_{k1}(s)|$ for each one of the eigenvalues $1+k^2$ in the vertical axis, versus $s$ in the horizontal axis on a semi-log scale. The graph suggests that
the order of approximation for all the eigenvalues changes at least two orders of magnitude as $s$ varies. The minimal residue in the approximation of the ground eigenvalue is achieved when $s=10$ and $s=12$. This corresponds to low contrast in the basis of $\cL(s)$. When the eigenvalue is multiple, however, the minimal residue is achieved by increasing the contrast in the basis. As this contrast increases, one conjugate pair will get closer to the real axis while the other will move away from it. The greater the $k$ is, the greater contrast is needed to achieve a minimal residue and the further away the conjugate pairs travel from each other.

This experiment suggests a natural extension for Theorem~\ref{thm4}. If only
$l<m$ members of $\cB$ are close to $\cL$, then only $l$ conjugate pairs on
$\Spec_2(A,\cL)$  will be close to the corresponding eigenvalue.

\subsection{Optimality of convergence to eigenvalues} \label{sturm_sec}
We saw in Example~\ref{nonoptimal} that the upper bound established in Theorem~\ref{thm4} is sub-optimal. We now examine this assertion from a practical perspective.

Let $A=-\partial_x^2+V$ acting on $\cH=L^2(0,\pi)$ where $V$ is a smooth real-valued bounded potential. Let
\[
    \Dom(A)=\{u\in H^2(0,\pi):u(0)=u(\pi)=0\},
\]
so that $A$ defines a self-adjoint operator semi-bounded.
Note that $\lambda\in\Spec(A)=\Spec_\dsc(A)$ if and only if $\lambda$ solves the Sturm-Liouville eigenvalue problem
$Au=\lambda u$ subject to homogeneous Dirichlet boundary conditions at $0$ and $\pi$.

Let $\Xi$ be an equidistant partition of $[0,\pi]$ into $n$ subintervals
$I_l=[x_{l-1},x_l]$ of length $h=\pi/n=x_{l}-x_{l-1}$. Let
\[
   \cL(h,k,r)=V_h(k,r,\Xi)=\{v\in C^k(0,\pi)\,:\,
   v\upharpoonright _{I_l}\in P_r(I_l),\, 1\leq l \leq n,\,
   v(0)=v(\pi)=0 \}
\]
be the finite element space generated by $C^k$-conforming elements
of order $r$ subject to Dirichlet boundary conditions at $0$ and $\pi$; \cite{MR0520174}. An implementation of standard interpolation error estimates
for finite elements combined with Theorem~\ref{thm4}
ensures the following.

\begin{lemma} \label{1Dconvergence}
 Let $A$, $\cH$ and $\cL(h,k,r)$ be as in the previous paragraphs. Let
\[
\Spec(A)=\{\lambda_1<\lambda_2<\ldots \}.
\]
Let $a_j=\frac14 \lambda_j+\frac34 \lambda_{j-1}$ and
$b_j=\frac14 \lambda_j+\frac34 \lambda_{j+1}$, where $\lambda_0=-\infty$.
For all $r>k\geq 3$, there exist a constant $c>0$, dependant on $j$, $k$
and $r$, but independent of $h$, such that
\[
    \Spec_2(A,\cL(h,k,r))\cap \D(a_{j},b_{j})=
    \{z_{hkr},\overline{z_{hkr}}\} \subset \D(\lambda_j-ch^{\frac{r-3}{2}},
    \lambda_j+ch^{\frac{r-3}{2}})
\]
for all $h>0$ sufficiently small.
\end{lemma}
\begin{proof}
Use the well-known estimate
\[
     \|v-v_h\|_{H^p(0,\pi)}\leq c h^{r+1-p}
\]
where $v_h\in \cL(h,k,r)$ is the finite element interpolant of
$v\in C^k\cap H^{r+1}_\Xi(0,\pi)$; \cite[Theorem~3.1.6]{MR0520174}. Note that all eigenvalues of $A$ are simple
and its eigenfunctions are $C^\infty$.
\end{proof}

Therefore each individual eigenvalue $\lambda_j$
is approximated by second order spectral points at a rate
$O(h^{\frac{r-3}{2}})$ for test
subspaces generated by a basis of $C^3$-conforming finite elements
of order $r> 4$. Due to the high regularity required on the approximating basis,
this results is only of limited practical use. In fact,
only $k\geq 1$ (and $r\geq 3$) is required for
$\cL(h,k,r)\subset \Dom(A)$. Simple numerical experiments confirm that the exponent predicted by Lemma~\ref{1Dconvergence} is  not optimal, see Figure~\ref{sturm}.

\begin{figure}[ht!]
\hspace{-1cm}
\begin{minipage}{6cm}\includegraphics[height=6cm]{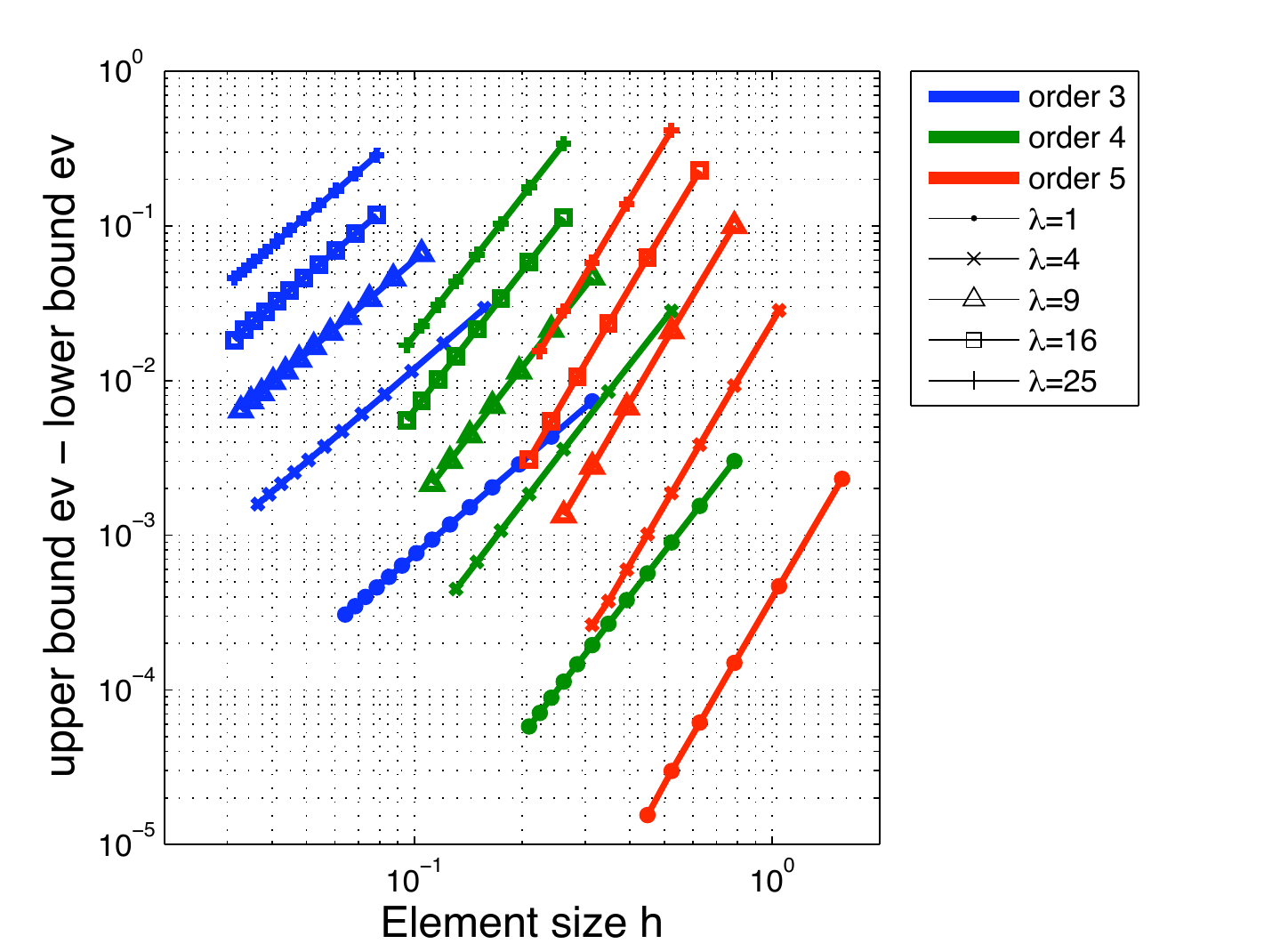}\end{minipage}
\hspace{2cm}
\begin{minipage}{4cm}\centerline{Table: $p$ such that $|\Im  z_{h1r}(j)|\sim h^p$} \ \newline \begin{tabular}{c|c|c|c}  $\lambda$  & $r=3$ & $r=4$ & $r=5$ \\
\hline
1 & 1.9979 & 2.9900 & 3.9849 \\
4 & 1.9983 & 2.9878 & 3.9191 \\
9 & 2.0023 & 2.9838 & 3.9235 \\
16 & 2.0099& 2.9764 & 3.9152\\
25 & 1.9962& 2.9673 & 3.8871
\end{tabular}\end{minipage}
\caption{Figure corresponding to Example~\ref{sturm1}. Left: Loglog plot of
$2|\Im  z_{h1r}(j)|$ versus element sizes. Right: Slopes of linear interpolations of these graphs.}
\label{sturm}
\end{figure}

\begin{example} \label{sturm1}
Suppose that $V=0$ so that $A=-\partial_x^2$. In this experiment we fix an equidistant partition $\Xi$ and let $\cL=\cL(h,1,r)$ be the space of Hermite elements of order $r=3,4,5$ satisfying Dirichlet boundary conditions in $0$ and $\pi$. We then find the conjugate pairs $\{z_{h1r}(j),\,
\overline{z_{h1r}(j)}\}\in \Spec_2(A,\cL)$ which are close
to $\lambda_j=j^2\in \Spec(A)$ for $j=1,2,3,4,5$.
On the left of Figure~\ref{sturm} we have depicted $2|\Im  z_{h1r}(j)|$
versus small values of $h$. On the right side we have tabulated
the slopes of linear fittings of the lines found.
\end{example}

This experiment suggests that the order of approximation of
$\lambda\in \Spec_\dsc(A)$ from its closest $z\in \Spec_2(A,\cL)$ is
\begin{equation} \label{conje}
    |z-\lambda|\sim \dist_{\Dom(A)} (\cB,\cL).
\end{equation}
Indeed, interpolation by Hermite elements of order $r$ has an $H^2$-error
proportional to $h^{r-1}$. The same convergence
rate is confirmed by Example~\ref{nonoptimal}.

\begin{example} \label{mathieu}
Let $V(x)=2\cos(2x)$ be the Mathieu potential. Let $\cL$ be as in Example~\ref{sturm1}. The exponents $p$ reported in Figure~\ref{mathieu_tab} further confirm \eqref{conje}.
\end{example}

\begin{remark} \label{comp_galer}
The error in the estimation of the eigenvalues of a one-dimensional elliptic
problem of order $2p$ by the Galerkin method using Hermite elements of order $r$ is proportional to $h^{2(r+1-p)}$; \cite[Theorem~6.1]{MR0443377}.
If $A$ is a second order differential operator, the quadratic eigenvalue problem  \eqref{qep} gives rise to a non-self-adjoint fourth order problem which is to be solved by a projection-type method. Thus \eqref{conje} is consistent with this estimate if we take into account the improved enclosure  \eqref{strauss}, see Section~\ref{improve}.
\end{remark}

\subsection{Improved accuracy} \label{improve}
Estimate \eqref{strauss} can be combined with
\eqref{eugene2} to provide improved \emph{a posteriori} enclosures
for $\lambda\in\Spec_\dsc(A)$. The key idea is to estimate an upper bound $a\!>\!\nu\!=\!\sup\{\tilde{\lambda}\!\in\!\Spec(A):\tilde{\lambda}\!<\!\lambda\}$, a lower bound
$b\!<\!\nu\!=\!\inf\{\tilde{\lambda}\!\in\!\Spec(A):\tilde{\lambda}\!>\!\lambda\}$ and
$z\in \Spec_2(A,\cL)$ such that \[(\Re z-|\Im z|,\Re z+|\Im z|)\cap \Spec(A)=\{\lambda\}.\]
Here $\mu$ and $\nu$ can be elements of the discrete or the essential spectrum of $A$.
The one-sided bounds $a$ and $b$ can be found from $\Spec_2(A,\cL)$ or by analytical means. If $b-a$ is sufficiently large and $|\Im z|$ is sufficiently small, \eqref{strauss}
improves upon \eqref{eugene2}. We illustrate this approach a practical settings.

\begin{example} \label{sturm_accurate}
Let $A=-\partial_x^2+2\cos(2x)$, $\cH=L^2(0,\pi)$ and $\cL=\cL(h,1,r)$ for $r=3,4,5$ be as in Example~\ref{mathieu}. In Figure~\ref{mathieu_tab} we have computed inclusions for the first five eigenvalues of $A$ by directly employing \eqref{eugene2} and by
the technique described in the previous paragraph. In the latter case,
we have found $a_j=a$
from the computed upper bound for $\lambda_{j-1}<\lambda_j$ ($a_1=-\infty$) using \eqref{eugene2}. Similarly for $b_j=b$.  These calculations can be compared
with those in \cite[Tables~1,2,3]{Arbenz:1983}. See also \cite[\S 7.4]{MR0010757}.
\end{example}

\begin{figure}[ht!]
\centerline{\begin{tabular}{c|c||ccccc}
$r$ & $j$ & 1 & 2 & 3 & 4 & 5 \\
\hline \hline
$3$ &  $n=48$ & & & & & \\
& $^\mathrm{upper}_\mathrm{lower}\mathrm{bounds}$ \eqref{eugene2}
& $-0.1^{087}_{117}$   & $3.9^{212}_{127}$  & $9.0^{619}_{334}$  & $1^{6.07}_{5.99}$   & $2^{5.12}_{4.92}$ \\
& $^\mathrm{upper}_\mathrm{lower}\mathrm{bounds}$ \eqref{strauss}
& $-0.11024^{881}_{936}$ & $3.91702^{928}_{122}$ &
$9.0477^{788}_{100}$ & $16.03^{322}_{277}$ & $25.0^{219}_{199}$ \\
\cline{2-2} & $n=10\!:\!5\!:\!50$  &&&&&\\ &
{\small $|\Im z_{h13}(j)|\sim h^p$} & $p\!\approx\!1.9915$ & $p\!\approx\!1.9847$ & $p\!\approx\!1.9790$ & $p\!\approx\!1.9682$ & $p\!\approx\!1.9532$ \\
\hline
$4$& $n=24$ & & & & & \\
& $^\mathrm{upper}_\mathrm{lower}\mathrm{bounds}$ \eqref{eugene2}
& $-0.110^{037}_{460}$ & $3.91^{767}_{637}$ & $9.04^{992}_{555}$ &
$16.0^{406}_{253}$ &  $2^{5.04}_{4.99}$ \\
& $^\mathrm{upper}_\mathrm{lower}\mathrm{bounds}$ \eqref{strauss}
& $-0.1102488^{170}_{282}$ & $3.917024^{878}_{690}$ & $9.0477^{401}_{385}$ &
$16.0329^{784}_{635}$ &  $25.020^{896}_{795}$ \\
\cline{2-2} & $n=9\!:\!2\!:\!19$ &&&&& \\ &  {\small $|\Im z_{h14}(j)|\sim h^p$}
& $p\!\approx\!2.9816$ & $p\!\approx\!2.9680$ & $p\!\approx\!2.9629$& $p\!\approx\!2.9486$ & $p\!\approx\!2.9238$ \\
\hline
$5$ & $n=12$ & & & & & \\
& $^\mathrm{upper}_\mathrm{lower}\mathrm{bounds}$ \eqref{eugene2}
& $-0.110^{151}_{345}$ & $3.91^{736}_{668}$ & $9.04^{886}_{661}$ &
$16.0^{372}_{286}$ & $25.0^{356}_{060}$ \\
& $^\mathrm{upper}_\mathrm{lower}\mathrm{bounds}$ \eqref{strauss}
& $-0.11024881^{697}_{932}$ & $3.917024^{800}_{750}$ & $9.047739^{506}_{080}$ &
$16.0329^{727}_{680}$ & $25.0208^{652}_{210}$ \\
\cline{2-2} & $n=9\!:\!1\!:\!12$   &&&&& \\ & {\small $|\Im z_{h15}(j)|\sim h^p$} &$p\!\approx\!3.9768$&$p\!\approx\!4.0214$&$p\!\approx\!3.9557$&$p\!\approx\!3.9432$& $p\!\approx\!3.9183$ \\
\end{tabular}}
\caption{\label{mathieu_tab} Enclosures for $\lambda_j$ in the case of
$A=-\partial_x^2+2\cos(2x)$ by direct application of \eqref{eugene2}
and by the improved \eqref{strauss}.
Here $\cL=\cL(\frac{\pi}{n},1,r)$.}
\end{figure}

We now consider an example from solid state physics to illustrate the case where $\mu$ and $\nu$ are not in the discrete spectrum.

\begin{example} \label{christalline}
Let $A=-\partial_x^2+\cos(x)-e^{-x^2}$ acting on $L^2(-\infty,\infty)$
where $\Dom(A)=H^2(-\infty,\infty)$. Then $A$ is a semi-bounded operator but now
$\Spec_\ess(A)$ consists of an infinite number of non-intersecting bands, separated by gaps, determined by the periodic part of the potential. The endpoints of these
bands can be found analytically. They correspond to the so called Mathieu characteristic values. The addition of a fast decaying perturbation gives rise to a non-trivial discrete spectrum.  Non-degenerate isolated eigenvalues can appear below the bottom of the essential spectrum or in the gaps between bands.

In Figure~\ref{christalline_eigenvalues} we report on computation of inclusions for the
first three eigenvalues in $\Spec_\dsc(A)$: $\lambda_1<\min \Spec_\ess(A)$, $\lambda_2$ in the first gap of the essential spectrum and $\lambda_3$ in the second gap.
No other eigenvalue is to be found in any of these gaps. Here $\cL$ is the space of Hermite elements of order 3 on a mesh of $20l$ segments of equal size $h=0.1$ in $[-l,l]$ subject to Dirichlet boundary conditions at $\pm l$. For the improved enclosure,  $a$ and $b$ are approximations of the endpoints of the gaps where the $\lambda_j$ lie.
\end{example}

\begin{figure}[ht!]
\begin{tabular}{c|c|cc|c}
 Eigenvalue &  $^\mathrm{upper}_\mathrm{lower}\mathrm{bounds}$ \eqref{eugene2} & $a$ &$b$ & $l$ \\
     & $^\mathrm{upper}_\mathrm{lower}\mathrm{bounds}$ \eqref{strauss} &&\\
\hline
$\lambda_1$ &  $-0.409^{535}_{718}$& $-\infty$ & $-0.378490$ & $25$\\
    &  $-0.409627^{318}_{588}$ &&\\
$\lambda_2$ &  $0.377^{791}_{494}$ & $-0.347670$ & $0.594800$& $50$\\
    &  $0.377633^{116}_{000}$ &&\\
$\lambda_3$ &  $1.18^{219}_{164}$&  $0.918058$ & $1.29317$  &$100$\\
    &   $1.18191^{726}_{629}$ &&
    \end{tabular}
\caption{\label{christalline_eigenvalues} Enclosures for the first three eigenvalues
of the perturbed periodic Schr\"odinger operator of Example~\ref{christalline}.
The numerical values of $a$ and $b$ are found in \cite{MR0167642} (see also \cite[Table~1]{bole}).}
\end{figure}

Since the eigenvectors of $A$ decay exponentially fast as $x\to \pm\infty$, the members of $\Spec_\dsc(A)$ are close to eigenvalues of the regular Sturm-Liouville problem $Au=\lambda u$ subject to $u(-l)=u(l)=0$ for sufficiently large $l<\infty$. The numerical method considered in this example does not distinguish between the $l=\infty$
and the large $l<\infty$ eigenvalue problem. For instance, the inclusion found for $\lambda_1$ in Figure~\ref{christalline_eigenvalues} is also an inclusion for the Dirichlet ground eigenvalue of $A\upharpoonright L^2(-25,25)$.

In the case of $\lambda_2$ and $\lambda_3$, which lie in
gaps of $\Spec_\ess(A)$, they should be close to high energy eigenvalues of the finite interval problem. Indeed, for the parameters considered in Figure~\ref{christalline_eigenvalues}, the inclusion for $\lambda_2$ is also an inclusion for the $17^{\text{th}}$ eigenvalue of $A\!\upharpoonright\! L^2(-50,50)$. Similarly that for
$\lambda_3$ is an inclusion for the $65^{\text{th}}$ eigenvalue of $A\!\upharpoonright \!L^2(-100,100)$.

\begin{remark}  \label{coment}
The error in the Galerkin approximation of the $j$-th eigenvalue of a regular Sturm-Liouville problem via Hermite elements of order 3 on an uniform mesh of element size $h$ is known to be $O(j^8h^6)$; \cite[\S6.2-(34)]{MR0443377}. Evidently the latter is only an upper bound, there might be high energy eigenvalues which are approximated accurately: such as the $17^{\text{th}}$ eigenvalue of $A\!\upharpoonright\! L^2(-50,50)$ or the $65^{\text{th}}$ eigenvalue of $A\!\upharpoonright\! L^2(-100,100)$ in the above example.
In practice it is not easy to take advantage of this observation as the Galerkin method also produces a considerable amount of spurious eigenvalues
in low-energy regions of the spectrum. These correspond to approximations of singular Weyl sequences of points in the essential spectrum of $A$.
\end{remark}
%%%%%%%%%%%%%%%%

\section{Accumulation points outside the real line}
\label{outside_real}

Since $\R$ is second countable and the essential spectrum of $A$ is closed, there always exists a family of open intervals $(a_j,b_j)\subset \R$ such that $\R\setminus \Spec_\ess(A)=\cup_{j=1}^\infty(a_j,b_j)$. Throughout this section we will be repeatedly referring to the following two $A$-dependent regions of the complex plane:
\[
     \B=\B_A=\bigcup_{j=1}^\infty \D(a_j,b_j) \qquad \text{and}
     \qquad \A=\A_A =\D[\inf \Spec_\ess A,\sup \Spec_\ess A]\setminus
     \B. 
\]
For any given set $\Omega\subset \C$ and $\varepsilon>0$ we will denote the open 
\emph{$\varepsilon$-neighbourhood} of $\Omega$ by
\[
     [\Omega]_\varepsilon=\{z\in\C:\dist(z,\Omega) <\varepsilon\}.
\] 

\begin{lemma} \label{bound_funct_sigma}
Let $\cN=\overline{\cN}\subseteq \B$ be compact and such that $\cN\cap \Spec(A)=\varnothing$. Let $(\cL_n)\in \Lambda_2$. There exists $s_\cN>0$ and $N>0$ such that
\[
     \sigma_{A,\cL_n}(z)\geq s_\cN \qquad \forall z\in \cN \text{ and }n\geq N.
\]
\end{lemma}
\begin{proof} $\B=\varnothing$  if and only if $\Spec(A)=\R$ and in this case there is nothing to proof, so we assume $\B\not=\varnothing$. Since $\cN$ is compact and
every covering of a compact set has a finite sub-covering there exists a finite family
$\cF$ such that $\cN\subset \cup_{j\in \cF}\D(a_j,b_j)$. Then we can decompose $\cN=\cup_{j\in \cF}\cN_j$ where each $\cN_j=\overline{\cN_j}$ are compact and such that
$\cN_j\subset \D(a_j,b_j)$. Since $\cN\cap \Spec(A)=\varnothing$, there exists
$\varepsilon>0$ such that $\cN_j\subset \D(a_j+\varepsilon,b_j-\varepsilon)$. By construction each interval $(a_j,b_j)$ can only
intersect $\Spec(A)$ at a finite number of isolated eigenvalues of finite multiplicity. 
Therefore the conclusion follows from lemmas~\ref{lem1} and \ref{lem2} applied in $(a_j+\varepsilon,b_j-\varepsilon)$ using analogous arguments as in the proof of Theorem~\ref{thm4}.
\end{proof}

Combining this lemma with a similar argument as in the proof of Corollary~\ref{thm1}, we immediately achieve the following theorem which generalises Corollary~\ref{cor1}.

\begin{theorem} \label{upper_inclusion}
For any $(\cL_n)\in \Lambda_1$,
\[
    \Spec_\dsc(A)\subseteq \lim_{n\to\infty} \Spec_2(A,\cL_n)\subseteq 
    \A\cup \Spec_\dsc(A).
\]
\end{theorem}

Let us now examine statements complementary  to this result. We begin with two
auxiliary lemmas.

\begin{lemma} \label{basis_essential_spectrum}
Let $c\in\Spec_\ess(A)$. For any given $m\in \N$ and $\delta>0$ there exists $\{x_l\}_{l=1}^m\subset \chi_{(c-\delta,c+\delta)}(A)$ such that
\begin{itemize}
\item[(i)] $\|x_l\|=1$ for all $l=1,\ldots,m$
\item[(ii)] $\langle A^p x_l,A^q x_{\tilde{l}} \rangle=0$ for all
$l\not=\tilde{l}$ and $p,q=0,1$
\item[(iii)] $Ax_l=cx_l+\hat{x}_l$ where $\|\hat{x}_l\|<\delta$ for any
$l=1,\ldots,m$.
\end{itemize}
\end{lemma}
\begin{proof}
If $c$ is an isolated point of the spectrum, then $c$ is an eigenvalue of infinite multiplicity and the proof is elementary. Otherwise, there exists $c_l\to c$ such that
$c_l\in \Spec(A)\setminus \{c\}$. By substituting $c_l$ by a sub-sequence if necessary, we can assume that $|c_l-c|=\delta(l)<\delta/2$ for $0<\delta(l+1)<\delta(l)$
such that
\[
    \left(c_l-\frac{\delta(l)}{2},c_l+\frac{\delta(l)}{2}\right)\cap \{c_l\}_{l=1}^\infty=\{c_l\}.
\]
By picking $x_l\in \chi _{\left(c_l-\frac{\delta(l)}{2},c_l+\frac{\delta(l)}{2}\right)}(A)$
for $l=1,\ldots,m$ with $\|x_l\|=1$ the desired conclusion is guaranteed.
\end{proof}

For the triplet $c_1\leq c_2\leq c_3$ denote 
\[
    \A(c_1,c_2,c_3)=\D[c_1,c_3]\setminus \big(\D(c_1,c_2)\cup \D(c_2,c_3)\big).
\]

\begin{lemma} \label{B}
Let $|c|\leq 1$, $0\leq \theta\leq \pi$ and $|c|\leq a \leq 1$. Let $z=ae^{i\theta}$
and
\[
    \beta_{\pm}=\frac12\left( \frac{a^2\pm c}{1 \pm c}\mp a\cos \theta\right).
\]   
If $z\in \A(-1,c,1)$, then $\beta_\pm\geq 0$.
\end{lemma}
\begin{proof} The proof involves straightforward trigonometric arguments. \end{proof}

\begin{lemma} \label{C}
Let $z\in \A(c_1,c_2,c_3)$. For $\delta>0$ let $x_1,x_2,x_3\in\Dom(A)$ be such that
\begin{itemize}
\item[(i)] $\|x_k\|=1$
\item[(ii)] $\langle A^p x_k,A^q x_{\tilde{k}} \rangle=0$ for
$k\not=\tilde{k}$ and $p,q=0,1$
\item[(iii)] $Ax_k=c_kx_k+\hat{x}_k$ where $\|\hat{x}_k\|<\delta$.
\end{itemize}
There exist $\alpha_k\in\R$ independent of $\delta$ such that if
$y=\sum_{k=1}^3\alpha_kx_k$, the polynomial
\[
    \langle Ay,Ay\rangle-2\lambda \langle Ay,y\rangle +\lambda^2 \langle y,y\rangle
\] 
has roots $\lambda_+ =z+O(\delta)$ and $\lambda_- =z+O(\delta)$.
\end{lemma}
\begin{proof}
The cases where $c_k=c_{\tilde{k}}$ for $k\not=\tilde{k}$ are easy, so we
assume that $c_1<c_2<c_3$. Moreover, without loss of generality
we can assume that $c_1=-1$, $c_2=c$ where $|c|<1$ and $c_3=1$.
Then $z$ and $\overline{z}$ are roots of the polynomial
\[
    a^2-2\lambda a \cos \theta +\lambda^2
\]
for $z=ae^{i\theta}$. If $y=\sum_{k=1}^3\alpha_kx_k$ we get
\begin{align*}
   \langle y,y\rangle &= \alpha_1^2+\alpha_2^2+\alpha_3^2 \\
   \langle Ay,y\rangle &= -\alpha_1^2+c\alpha_2^2+\alpha_3^2 +O(\delta) \\
   \langle Ay,Ay\rangle &= \alpha_1^2+c^2\alpha_2^2+\alpha_3^2 +O(\delta).
\end{align*}
Finally note that the solution of the system
\begin{align*}
   \alpha_1^2+\alpha_2^2+\alpha_3^2&=1 \\
   -\alpha_1^2+c\alpha_2^2+\alpha_3^2 &=a\cos \theta \\
    \alpha_1^2+c^2\alpha_2^2+\alpha_3^2 &= a^2
\end{align*}
is easily found to be $\alpha_1=\sqrt{\beta_+}$,
$\alpha_2=\sqrt{\frac{1-a^2}{1-c^2}}$ and $\alpha_3=\sqrt{\beta_-}$
where $\beta_{\pm}$ are the expressions found in Lemma~\ref{B}.  
\end{proof}

\begin{theorem} \label{lower_inclusion}
Let $\cN=\overline{\cN}\subseteq \A$ be compact. For all $\varepsilon>0$
there exists $\cL_\varepsilon\subset \Dom(A)$ such that 
\begin{equation} \label{chain_inc}
    \cN\subset  
    \left[\Spec_2(A,\cL_\varepsilon)\right]_\varepsilon
    \subset [\cN]_{2\varepsilon}.
\end{equation}
\end{theorem}
\begin{proof} Let $\cN$ and $\varepsilon>0$ be fixed.
Since $\cN$ is compact, there exists $\{z_j\}_{j=1}^n
\subset \cN$ such that $\cN\subset \left[(z_j)_{j=1}^n \right]_{\varepsilon/4}
\subset [\cN]_{\varepsilon/2}$. The proof will be completed if $\cL_\varepsilon\subset \Dom(A)$ is such that
\begin{equation} \label{finite_case}
     \left[ \{z_j\}_{j=1}^n \right]_{\varepsilon/4}\subset
     \left[\Spec_2(A,\cL_\varepsilon)\right]_\varepsilon \subset
      \left[ \{z_j\}_{j=1}^n \right]_{\varepsilon/2}.
\end{equation}
 Below we will choose the parameter $\delta>0$ small enough.

Let $\{c_{jk}\}_{k=1,j=1}^{k=3,j=n}\subset \Spec_\ess(A)$ be such that
$z_j\in \A(c_{1j},c_{2j},c_{3j})$ for all $j=1,\ldots,n$. By virtue of Lemma~\ref{basis_essential_spectrum} there exists a family of vectors $\{x_{jk}\}_{k=1,j=1}^{k=3,j=n}\subset \Dom(A)$ such that 
\begin{itemize}
\item[(i)] $\|x_{kj}\|=1$ for all $k=1,2,3$ and $j=1,\ldots,n$
\item[(ii)] $\langle A^p x_{kj},A^q x_{\tilde k \tilde j} \rangle=0$ for all
$(k,j)\not=(\tilde{k},\tilde j)$ and $p,q=0,1$
\item[(iii)] $Ax_{kj}=c_{kj}x_{kj}+\hat{x}_{kj}$ where 
$\|\hat{x}_{kj}\|<\delta$ for any $k=1,2,3$ and $j=1,\ldots,n$.
\end{itemize}
Let $\alpha_{kj}=\alpha_k$ be as in Lemma~\ref{C} for $z=z_j$.
Choosing $\delta>0$ small enough and defining
\[
    y_j=\sum_{k=1}^3 \alpha_{kj} x_{kj} \qquad \text{and} \qquad
    \cL_\varepsilon=\Span\{y_1,\ldots,y_n\},
\]  
yields \eqref{finite_case} and hence \eqref{chain_inc}. 
\end{proof}

This result has two straightforward consequences. Given any compact subset $\cN=\overline{\cN}\subseteq \A$, there exists $\{\cL_n\}\subset \Dom(A)$ such that $\cN=\lim_{n\to \infty} \Spec_2(A,\cL_n)$. Evidently $(\cL_n)$ might fall outside $\Lambda_0$ in general. On the other hand, we can always find $(\cL_n)\in \Lambda_1$ such that
\[
    \lim_{n\to \infty} \Spec (A,\cL_n)=\A \cup \Spec_\dsc(A).
\]

%%%%%%%%%%%%%%%%%

\section{Acknowledgements}
We are grateful to Matthias Langer for fruitful discussions.
L.~Boulton wishes to thank the hospitality of Ceremade where part of this research was carried out. M.~Strauss gratefully acknowledges the support of EPSRC grant no. EP/E037844/1.

\bibliographystyle{plain}
\bibliography{biblio}

\end{document}